 \documentclass[jip]{degruyter-journal-a}      

\usepackage{graphicx}
\usepackage{algorithm}
\usepackage[noend]{algpseudocode}
\usepackage{amssymb}
\usepackage{amsmath}

\def\N{\mathbb{N}}

\usepackage{color}
\def\revision#1{{#1}}
\newcommand{\revmarg}[1]{}

\title{The Ivanov regularized Gauss-Newton method in Banach space with an a posteriori choice of the regularization radius}
\headlinetitle{The Ivanov regularized Gauss-Newton method in Banach space}

\lastnameone{Kaltenbacher}
\firstnameone{Barbara}
\nameshortone{B. Kaltenbacher}
\addressone{Alpen-Adria-Universit\"at, Universit\"atsstr. 65-67, A-9020 Klagenfurt}
\countryone{Austria}
\emailone{barbara.kaltenbacher@aau.at}

\lastnametwo{Klassen}
\firstnametwo{Andrej}
\nameshorttwo{A. Klassen}
\addresstwo{University of Duisburg-Essen, Thea-Leymannstra\ss e 9, 45127 Essen}
\countrytwo{Germany}
\emailtwo{andrej.klassen@uni-due.de}

\lastnamethree{Previatti de Souza}
\firstnamethree{Mario Luiz}
\nameshortthree{M.L. Previatti de Souza}
\addressthree{Alpen-Adria-Universit\"at, Universit\"atsstr. 65-67, A-9020 Klagenfurt}
\countrythree{Austria}
\emailthree{mario.previatti@aau.at}

\abstract{In this paper we consider the iteratively regularized Gauss-Newton method, where regularization is achieved by Ivanov regularization, i.e., by imposing a priori constraints on the solution. We propose an a posteriori choice of the regularization radius, based on an inexact Newton / discrepancy principle approach, prove convergence and convergence rates under a variational source condition as the noise level tends to zero, and provide an analysis of the discretization error. Our results are valid in general, possibly nonreflexive Banach spaces, including, e.g., $L^\infty$ as a preimage space.
The theoretical findings are illustrated by numerical experiments.
}



\researchsupported{supported by the FWF under grants I2271 and P30054 and by the DFG under grant Cl 487/1-1,
as well as partially by the Karl Popper Kolleg ``Modeling-Simulation-Optimization'', funded by the Alpen-Adria-Universit\" at Klagenfurt and by the Carin\-thian Economic Promotion Fund (KWF).
}

\acknowledgments{The authors wish to thank Christian Clason, University of Duisburg-Essen for valuable discussions and support with the implementation.
\revision{
Moreover, the authors wish to thank all reviewers for their very careful reading of the manuscript and their detailed reports with valuable comments and suggestions that have led to an improved version of the paper.
}
}

\dedicated{}

\begin{document}

\section{Introduction}
Consider an inverse problem given as a nonlinear ill-posed operator equation
\begin{equation}\label{Fxg}
F(u)=g,
\end{equation} where the possibly nonlinear operator $F:\mathcal{D}(F)\subset \mathcal{U}\longrightarrow \mathcal{Y}$ with domain $\mathcal{D}(F)$ maps between real Banach spaces $\mathcal{U}$ and $\mathcal{Y}$.
The task is to recover $u$ (or actually an approximation of it), given noisy observations $g^\delta$ of $g$. Due to the ill-posedness of \eqref{Fxg}, i.e., the lack of continuous invertibility of $F$, the problem needs to be regularized (see, e.g., \cite{BakKok04,EnglHankeNeubauer,Groetsch,KalNeuSch08,Kirsch,Louis,MorozovBuch,SKHK12,TikhonovArsenin,VainikkoVeterennikov,VasinAgeev}, and the references therein).

Throughout this paper we will assume that an exact solution $\revision{u^*} \in \mathcal{D}(F)$ of \eqref{Fxg} exists, i.e., $F(\revision{u^*})=g$, and that the deterministic noise level $\delta$ in the estimate 
\begin{equation}\label{noise}
\|g-g^\delta\|\leq \delta,
\end{equation} is known.

Ivanov regularization, i.e., the option of using a prior bounds for regularizing ill-posed problem has been known for a long time already, partly also as method of quasi solutions \cite{DombrovskajaIvanov65,Ivanov62,Ivanov63,IvanovVasinTanana02,SeidmanVogel89} and has been revisited and further analyzed recently 
\cite{ClasonKlassen18,KK17,LorenzWorliczek13,NeubauerRamlau14}. 

Here, we are particularly interested in the IRGN-Ivanov method, i.e., we define Newton type iterates $u_{k+1}^\delta$ in an Ivanov regularized way as
\begin{equation}\label{ivanov}
\begin{aligned}
u_{k+1}^\delta(\rho) \in \mbox{argmin}_{u\in\mathcal{D}(F)} \|F'(u_k^\delta)(u-u_k^\delta)+F(u_k^\delta)-g^\delta\| \\
\mbox{ such that } \mathcal{R}(u)\leq \rho_k,
\end{aligned}
\end{equation} 
with stopping index $k_*=k_*(\delta,g^\delta)$ according to the discrepancy principle
\begin{equation}\label{discprinc}
k_*=k_*(\delta,g^\delta)=\min \{k\in\mathbb{N}_0 : \|F(u^\delta_k)-g^\delta\| \leq \tau \delta\},
\end{equation} 
for $\tau>1$ chosen sufficiently large independently of $\delta$. 
(This includes the theoretical value $k_*=\infty$ in case the set over which the minimum is taken is empty, which we will anyway exclude later on, though). 
Regularization is defined by a proper and convex functional $\mathcal{R}$.
The example 
\begin{equation}\label{Rnorm}
\mathcal{R}(u)=\|u-u_0\|^p 
\end{equation}
with some norm defined on $\mathcal{D}(F)$, some $p>1$ and some a priori guess $u_0$ is frequently used in practice and will also be focused on in some of our results (e.g., Proposition \ref{prop:convR} below). 
In case of $\mathcal{R}$ being defined by the $L^\infty$ norm, the inequality in \eqref{ivanov} becomes a pointwise bound constraint and efficient methods for such mimimization problems can be 
\revision{
employed,
}
\revmarg{1, 1.}
\revmarg{2, 12.}
\revmarg{3, 4.}
cf., e.g., \cite{ClasonKlassen18,compminIP} in the context of ill-posed problems.
Since sums of convex functions are again convex, $\mathcal{R}$ may be composed as a sum of different terms. For instance, additional a priori information on $u$ might be incorporated in a strict sense by adding in $\mathcal{R}$ the indicator function of some convex set $\mathcal{C}$. 

The method \eqref{ivanov}, \eqref{discprinc} has already been considered in \cite{KaltenbacherPreviatti18}, with a choice of the  regularization radii $\rho_k$ that refers to the value of the regularization functional at the exact solution, more precisely, $\rho_k\equiv \rho\geq\mathcal{R}(u^*)$. Since this information might not be available in some practical applications, we here propose a special choice of $\rho_k$ that does not require the knowledge of $\mathcal{R}(u^*)$.
Instead, for fixed $0 < \theta < \Theta < 1$ 
and $k\leq k_*$, the regularization parameter $\rho=\rho_k$ is chosen in an {\em a posteriori} fashion according to the inexact Newton type rule
\begin{equation}\label{choice}
\theta\|F(u_k^\delta)-g^\delta\|\leq \|F'(u_k^\delta)(u_{k+1}^\delta(\rho)-u_k^\delta)+F(u_k^\delta)-g^\delta\|\leq \Theta\|F(u_k^\delta)-g^\delta\|,
\end{equation} 
(see also \cite{HankeLevMar,Rieder} in the context of the Levenberg-Marquardt method in Hilbert spaces)
provided 
\begin{equation}\label{caseA}
\|F'(u_k^\delta)(u_0-u_k^\delta)+F(u_k^\delta)-g^\delta\|\ge\Theta\|F(u_k^\delta)-g^\delta\|
\end{equation}
In this case we set $u_{k+1}^\delta := u_{k+1}^\delta (\rho_k)$.
\revision{
We will show that \eqref{caseA} is in fact always satisfied.
}
\revmarg{1, 8.}
\revmarg{2, 1.}
As compared to \cite{KaltenbacherPreviatti18}, where $\rho_k\geq\mathcal{R}(u^*)$, we will here prove that the reguarization radius chosen according to \eqref{choice} satisfies $\rho_k\leq\mathcal{R}(u^*)$, cf. \eqref{mon}, which implies that the regularization acts in a more restrictive, hence stronger stabilizing way. 

We point out that the proof of well-definedness of the regularization radius will strongly rely on recent results from \cite{ClasonKlassen18}.
Moreover, we emphasize the fact that due to the non-additive structure of regularization here, the analysis done so far (e.g. \cite{BakKok04,KalNeuSch08,SKHK12}) does not apply here. For the same reason, we also did not find a possibility to extend the Levenberg-Marquardt method (e.g., \cite{HankeLevMar,JinLevMar}) to the Ivanov regularized setting.

\revmarg{1, 8.}
\revmarg{2, 1.}
\begin{algorithm}
\caption{Ivanov-Iteratively Regularized Gauss Newton\label{IvanovAlg}}
\begin{algorithmic}[1]
\State Choose constants $c_{tc}, \tau, \theta, \Theta$ according to \eqref{ctc}, \eqref{discprinc} and \eqref{theta}, respectively.
\revision{
\State Choose starting point $u_0$. Set $k=0$.
\If{$\|F(u_k^\delta)-g^\delta\|\leq\tau\delta$}
\State $u_{k_*(\delta)}^\delta := u_0$ 
\Else
\While{$\|F(u_k^\delta)-g^\delta\|>\tau\delta$}
\State compute minimizer $u_{k+1}^\delta(\rho)$ of \eqref{ivanov} according to the rule \eqref{choice} for $\rho$ 
\State $u_{k+1}^\delta := u_{k+1}^\delta(\rho)$
\State $k:=k+1$
\EndWhile
\State $u_{k_*(\delta)}^\delta := u_k^\delta$
\EndIf
}
\end{algorithmic}
\end{algorithm}

The remainder of this paper is organized as follows.
In Section \ref{sec:conv} we provide a convergence analysis in the sense of a regularization method, along with rates under variational source conditions, which is carried over to the discretized setting under certain accuracy requirements in Section \ref{sec:discr}. Section \ref{sec:numexp} is devoted to implementation details and numerical experiments. The final Section \ref{sec:concl} contains a summary and an outlook.

\section{Convergence analysis}\label{sec:conv}

For analyzing convergence of this method, we impose the following conditions
\begin{assumption}\label{ass1}
Assume that a solution $u^\dagger$ of \eqref{Fxg} exists such that $R:=\mathcal{R}(u^\dagger) < \infty$.\\
Moroever, let topologies $\mathcal{T}_{\mathcal{U}}$ on $\mathcal{U}$ and $\mathcal{T}_{\mathcal{Y}}$ on $\mathcal{Y}$ exist such that 
\begin{enumerate}
\item[(a)] $\mathcal{R}\geq0$ is proper, convex, and $\mathcal{T}_{\mathcal{U}}$-lower semicontinuous with $\mathcal{R}(u)=0 \ \Longleftrightarrow \ u=u_0$;
\item[(b)] for all $0 \leq r \leq\mathcal{R}(u^\dagger)$, the sublevel set $B_r$ defined by 
\begin{equation}\label{Br}
B_r=\{ u \in \mathcal{D}(F) : \mathcal{R}(u) \leq r \} 
\end{equation} 
is compact with respect to $\mathcal{T}_{\mathcal{U}}$;
\item[(c)] bounded sets in $\mathcal{Y}$ are $\mathcal{T}_{\mathcal{Y}}$-compact and the norm in $\mathcal{Y}$ is $\mathcal{T}_{\mathcal{Y}}$-lower semicontinuous;
\item[(d)] for all $u \in B_R$, the linear operator $F'(u)$ is $\mathcal{T}_{\mathcal{U}}$-to-$\mathcal{T}_{\mathcal{Y}}$ closed, i.e., for any sequence $(u_k)_{k\in\N}\subseteq \mathcal{U}$,
\[
\Bigl(u_k \stackrel{\mathcal{T}_{\mathcal{U}}}{\longrightarrow} \bar{u} \mbox{ and } F'(u)u_k\stackrel{\mathcal{T}_{\mathcal{Y}}}{\longrightarrow} g\Bigr)
\ \Longrightarrow \  F'(u)\bar{u}=g\,.
\]
\end{enumerate} 
\end{assumption}

In case of the regularization functional being defined by the norm on the space $\mathcal{U}$ \eqref{Rnorm}, compactness of sublevel sets typically holds in the sense of weak or weak* sequential compactness (if $\mathcal{U}$ is reflexive or the dual of a separable space, respectively) provided $\mathcal{D}(F)$ is closed with respect to this topology as well. For example, $\mathcal{U}$ may be a Lebesgue $L^p(\Omega)$ or Sobolev spaces $W^{s,p}(\Omega)$ with summability index $p\in(1,\infty]$, the space of regular Radon measures $\mathcal{M}(\Omega)=C_b(\Omega)^*$, or the space of functions with bounded total variation $BV(\Omega)$ for some domain $\Omega$.
Thus, the topologies $\mathcal{T}_{\mathcal{U}}$, $\mathcal{T}_{\mathcal{Y}}$ will typically be weak or weak* topologies, or possibly also strong topologies arising from compact embeddings.
\revision{
Note that in general, we do not make explicit use of any norm on $X$ here, but actually only use the $\mathcal{T}_{\mathcal{U}}$ from Assumption \eqref{ass1} and the structure induced by $\mathcal{R}$ bounded sets $B_r$.
}

As a constraint on the nonlinearity of the forward operator $F$, we impose the tangential cone condition
\begin{eqnarray}\label{ctc}
\|F(\bar{u})-F(u)-F'(u)(\bar{u}-u)\|\leq c_{tc}\|F(\bar{u})-F(u)\|, && \forall \bar{u},u \in B_R,
\end{eqnarray} for $c_{tc} < 1/3$. 
Here, $B_R$ is defined as in \eqref{Br} and 
\revision{
$R=\mathcal{R}(u^\dagger)$ as in Assumption \ref{ass1}.
}
Also, $F'(u)\in L(\mathcal{U},\mathcal{Y})$ is some linearization of $F$ (not necessarily a G\^{a}teaux or Fr\'{e}chet derivative of $F$; the only requirements on $F'$ are \eqref{ctc} and $F'(u)\in L(\mathcal{U},\mathcal{Y})$ for all $u \in B_R$.
\revision{
In order to verify \eqref{ctc} in case $\mathcal{R}$ in the definition of $B_R$ is defined by a norm \eqref{Rnorm}, one typically restricts the radius $R$ (which basically determines the size of the constant $c_{tc}$) to be small, which will usually be possible also here, assuming closeness of an exact solution $u^*$ to $u_0$. 
In case $R$ needs to be large, since $\mathcal{R}(u^*)$ is large, closeness of $\bar{u},u \in B_R$ to each other can be enforced by considering $\mathcal{R}$ as the sum of some regularization functional and the indicator function of a sufficiently small neighborhood (possibly with respect to a different topology than the one used for regularization, and independent of the regularization radius) around $u^*$.
}
\revmarg{2, 2.}

Moroever, we assume that the thresholds for the choice of $\rho_k$ are chosen such that 
\begin{equation}\label{theta}
0< c_{tc}+\frac{1+c_{tc}}{\tau} < \theta < \Theta < 1-2c_{tc} <1\,.
\end{equation}

\begin{theorem}\label{theo_conv}
Let Assumption \ref{ass1} as well as condition \eqref{ctc} on $F$ be satisfied. 
Moreover, consider a family of data $(g^\delta)_{\delta >0}$ satisfying \eqref{noise} and let, for each $\delta$, $g^\delta$, the stopping index be defined by \eqref{discprinc} with $\tau > \frac{1+c_{tc}}{1-3c_{tc}}$, 
\revision{
which enables a choice of $\theta$, $\Theta$ such that \eqref{theta} holds.
}
\revmarg{2, 3.}

\begin{enumerate}
\item[\revision{(i)}] For any 
$\rho\in(0,R]$,
and $k<k_*$, the iterate $u_{k+1}^\delta(\rho)$ is well defined by \eqref{ivanov};
\item[(ii)] 
\revision{
For any $k<k_*$, under condition \eqref{caseA}, and if $u_k^\delta\in B_R$, 
\begin{enumerate}
\item[(a)] the choice of $\rho_k$ according to \eqref{choice} is well defined  
\item[(b)] for any solution $u^*\in B_R$ of \eqref{Fxg}, the estimate
\begin{equation}\label{mon}
 \mathcal{R}(u_{k+1}^\delta) \leq \rho_k\leq\mathcal{R}(u^*).
\end{equation} holds. 
\item[(c)] the estimate 
\begin{equation}\label{contr}
\|F(u_{k+1}^\delta)-g^\delta\|\leq q \|F(u_k^\delta)-g^\delta\|
\end{equation}
holds with some  $q < 1$ independent of $k$ and $\delta$ (cf. \eqref{q}), 
\item[(d)] 
\eqref{caseA} is satisfied with $k$ replaced by $k+1$.
\end{enumerate}
In particular, for fixed $\delta$, $g^\delta$, by induction and the fact that \eqref{caseA} holds for $k=0$, condition \eqref{caseA} remains valid throughout the iteration, 
and the iterates according to Algorithm \ref{IvanovAlg} are well-defined and remain in $B_R$. 
}
\revmarg{1, 8.}
\revmarg{2, 1.}
\item[(iii)] The stopping index $k_*$ according to \eqref{discprinc} is finite.
\item[(iv)] We have $\mathcal{T}_{\mathcal{U}}$-subsequential convergence as $\delta \rightarrow 0$, i.e., $(u_{k_*(\delta,g^\delta)}^\delta)_{\delta >0}$ has a $\mathcal{T}_{\mathcal{U}}$-convergent subsequence and the limit of any $\mathcal{T}_{\mathcal{U}}$-convergence subsequence solves \eqref{Fxg}. If the solution $u^*$ of \eqref{Fxg} is unique in $B_R$, then $u_{k_*(\delta,g^\delta)}^\delta \stackrel{\mathcal{T}_{\mathcal{U}}}{\longrightarrow} u^* $ as $\delta \rightarrow 0$.
\end{enumerate}
\end{theorem}

\begin{proof}

The existence  \revision{(i)} of a minimizer $u_{k+1}^\delta (\rho)$ of \eqref{ivanov} for fixed $k$, $u_k^\delta$, $g^\delta$ and $\rho$ follows by the direct method of calculus of variations (note that the setting differs from the one in \cite{KaltenbacherPreviatti18} in that we do not assume admissibility of $u^*$ here):
The cost functional
\begin{equation}\label{cost}
J_k(u)=\|F'(u_k^\delta)(u-u_k^\delta)+F(u_k^\delta)-g^\delta\|
\end{equation} is bounded from below and the admissible set $\mathcal{U}^{ad} = B_\rho$ is nonempty (this follows from the fact that $\rho > 0$ and $u_0 \in \mathcal{U}^{ad}$). Hence, there is a minimizing sequence $(u^l)_{l\in\mathbb{N}} \subseteq \mathcal{U}^{ad}$ with $\lim_{l\rightarrow \infty} J_k(u^l)=\inf_{u\in \mathcal{U}^{ad}} J_k(u)$.
By $\mathcal{T}_{\mathcal{U}}$-compactness of $B_\rho$, 
the sequence $(u^l)_{l\in\mathbb{N}}$ has a $\mathcal{T}_{\mathcal{U}}$-convergent subsequence $(u^{l_n})_{n\in\mathbb{N}}$ with limit $\bar{u}\in B_\rho$.
Boundedness of $J_k(u^l)$ yields $\mathcal{T}_{\mathcal{Y}}$ convergence of another (not relabelled) subsequence of 
$F'(u_k^\delta)u^{l_n}$ to some $g\in \mathcal{Y}$, which by the assumed $\mathcal{T}_{\mathcal{U}}$-to-$\mathcal{T}_{\mathcal{Y}}$ closedness of $F'(u_k^\delta)$ coincides with $F'(u_k^\delta)\bar{u}$. 
$\mathcal{T}_{\mathcal{Y}}$-lower semicontinuity of the norm in $\mathcal{Y}$ yields
\begin{equation*}
J_k(\bar{u}) \leq \liminf_{n\rightarrow \infty} J_k(u^{l_n}) = \inf_{u\in\mathcal{U}^{ad}}J_k(u) \mbox{ and } \bar{u}\in \mathcal{U}^{ad},
\end{equation*}
which implies that $\bar{u}$ is a minimizer of \eqref{ivanov}.

For proving \revision{(ii)}, fix $k$ such that $\|F(u_k^\delta)-g^\delta\| > \tau\delta$ and 
\revision{
assume \eqref{caseA} to hold. 
}
\revmarg{1, 8.}
\revmarg{2, 1.}
Defining $d(\rho,g):=\min\{\|F'(u_k^\delta)u-g\|\, :\, u\in B_\rho\}$ and using the notation $\bar{g}^\delta_k := g^\delta+F'(u_k^\delta)u_k^\delta-F(u_k^\delta)$, we have
\begin{equation*}
d(0,\bar{g}_k^\delta)=\|F'(u^\delta_k)(u_0-u_k^\delta)+F(u_k^\delta)-g^\delta\| \geq \Theta \|F(u_k^\delta)-g^\delta\|.
\end{equation*}
On the other hand, by \eqref{ctc}, we have, for any solution $u^*\in B_R$ of \eqref{Fxg},
\begin{eqnarray}\nonumber
d(\mathcal{R}(u^*),\bar{g}_k^\delta)&\leq&\|F'(u^\delta_k)(u^*-u_k^\delta)+F(u_k^\delta)-g^\delta\| \\
&\leq& c_{tc}\|F(u_k^\delta)-g^\delta\|+(1+c_{tc})\delta\nonumber\\
&\leq& \left(c_{tc}+\frac{1+c_{tc}}{\tau}\right)\|F(u_k^\delta)-g^\delta\| \label{m1}
\leq \theta \|F(u_k^\delta)-g^\delta\|.
\end{eqnarray} 
Thus, using continuity of the distance mapping $\rho \mapsto d(\rho,g)$ for 
\revision{
$\rho\in [0,\mathcal{R}(x^*)]$,
}
\revmarg{2, 12.}
see \cite{ClasonKlassen18}, together with the Intermediate Value Theorem, we have existence of 
\begin{equation} \label{rhokin0Ru}
\rho=\rho_k \in [0,\mathcal{R}(u^*)]
\end{equation} 
such that \eqref{choice} holds.
\revision{
To show that \eqref{rhokin0Ru} holds for any $\rho_k$ satisfying \eqref{choice},
}
\revmarg{1, 2.}
observe that the lower bound in \eqref{choice} means that
\begin{equation}\label{m2}
d(\rho_k,\bar{g}^\delta_k) \geq \theta\|F(u_k^\delta)-g^\delta\|,
\end{equation} 
hence by the monotone decrease of $\rho \rightarrow d(\rho,g)$,  (cf. \cite{ClasonKlassen18},) a combination of \eqref{m1} and \eqref{m2} implies
\begin{equation*}
d(\mathcal{R}(u^*),\bar{g}^\delta_k)\leq d(\rho_k,\bar{g}_k^\delta) \Rightarrow \mathcal{R}(u^*) \geq \rho_k\,.
\end{equation*} 
Thus, \eqref{mon} holds and therefore the iterates remain in $B_R$.

The residuals can be estimated as follows
\begin{eqnarray}\nonumber
&& (1-c_{tc})\|F(u_{k+1}^\delta)-g^\delta\|-c_{tc}\|F(u_k^\delta)-g^\delta\| \leq \\
\label{con}
&&\|F'(u^\delta_k)(u_{k+1}^\delta-u_k^\delta)+F(u_k^\delta)-g^\delta\| \leq \Theta \|F(u_k^\delta)-g^\delta\|, 
\end{eqnarray} 
which implies
\begin{equation*}
\|F(u_{k+1}^\delta)-g^\delta\| \leq q\|F(u^\delta_k)-g^\delta\|, 
\end{equation*}
where 
\begin{equation}\label{q}
q=\dfrac{\Theta +c_{tc}}{1-c_{tc}} <1.
\end{equation} 
\revision{
Thus we have $\|F(u_{k+1}^\delta)-g^\delta\|\leq q\|F(u_0)-g^\delta\|$ with $q$ as in \eqref{q}.
Using \eqref{ctc} and \eqref{q} we therefore get
\begin{equation}\label{caseAkp1}
\begin{aligned}
&\|F'(u_{k+1}^\delta)(u_0-u_{k+1}^\delta)+F(u_{k+1}^\delta)-g^\delta\|\\
&\geq\|F(u_0)-g^\delta\|-\|F'(u_{k+1}^\delta)(u_0-u_{k+1}^\delta)+F(u_{k+1}^\delta)-F(u_0)\|\\
&\geq \|F(u_0)-g^\delta\|-c_{tc}\|F(u_0)-F(u_{k+1}^\delta)\|\\
&\geq (1-c_{tc})\|F(u_0)-g^\delta\|-c_{tc}\|F(u_{k+1}^\delta)-g^\delta\|\\
&\geq ((1-c_{tc})q^{-1} -c_{tc})\|F(u_{k+1}^\delta)-g^\delta\|
\geq\Theta\|F(u_{k+1}^\delta)-g^\delta\|.
\end{aligned}
\end{equation} 
}
\revmarg{1, 8.}
\revmarg{2, 1.}

To see (iii), observe that from \eqref{contr} it follows that $\|F(u_k^\delta)-g^\delta\| <\tau\delta$ as soon as  
\begin{equation*}
k \geq (\log 1/q)^{-1}\left(\log \|F(u_0)-g^\delta\| -\log \tau \delta \right) =: \bar{k}_*(\delta) \geq k_*(\delta, g^\delta),
\end{equation*}
hence the stopping index defined by \eqref{discprinc} is indeed finite. 

Now $(iv)$ follows from $(ii)$ by standard arguments and our assumption on $\mathcal{T}_{\mathcal{U}}$-compactness of $B_R$.
In fact, let $(\delta_n)_{n\in\mathbb{N}}$ be an arbitrary sequence converging to zero. By \eqref{discprinc} and the fact that $0 \leq \mathcal{R}(u^{\delta_n}_{k_*(\delta_n ,g^{\delta_n})}) < R$, Assumption 1 yields existence of a $\mathcal{T}_{\mathcal{U}}$-convergent subsequence $(u^l)_{l\in\mathbb{N}}:=(u^{\delta_{n_l}}_{k_*(\delta_{n_l},g^{\delta_{n_l}})})_{l\in\mathbb{N}}$ with limit $\bar{u}$, satisfying $\mathcal{R}(\bar{u}) \leq \liminf_{l\rightarrow\infty}\mathcal{R}(u^l)$. Using \eqref{noise} and \eqref{discprinc}, we have
existence of a (not relabeled) subsequence $F(u^l)$ with $\mathcal{T}_{\mathcal{Y}}$ limit $y$, which by Assumption \ref{ass1} (c) coincides with $F(\bar{u})$, hence
\begin{equation*}
\|F(\bar{u})-g\|\leq \liminf_{l\rightarrow \infty} \|F(u^l)-g\|\leq \liminf_{l\rightarrow \infty}(\|F(u^l)-g^{\delta_{n_l}}\|+\delta_{n_l})=0.
\end{equation*}
Thus, $\bar{u}$ solves \eqref{Fxg}.

In case of uniqueness, convergence of the whole sequence follows by a subsequence-subsequence argument.
\end{proof}

While Theorem \ref{theo_conv} only gives weak (i.e., $\mathcal{T}_{\mathcal{U}}$-subsequential) convergence for general regularization functionals, more can be said in the special but practically relevant case that $\mathcal{R}$ is defined by a norm.

\begin{proposition} \label{prop:convR}
Under the assumptions of Theorem \ref{theo_conv}, with $\mathcal{R}$ defined by a norm \eqref{Rnorm}, we have, for $u_{k_*{(\delta,g^\delta)}}^\delta$ defined by Algorithm \ref{IvanovAlg}, that 
\begin{eqnarray}\label{convR}
\mathcal{R}(u_{k_*{(\delta,g^\delta)}}^\delta)&\rightarrow &
\inf\{\mathcal{R}(u^*)\, : \, u^*\in B_R \mbox{ solves \eqref{Fxg}} \}\\
&=&\min\{\mathcal{R}(u^*)\, : \, u^*\in B_R \mbox{ solves \eqref{Fxg}} \}=:R^*
\mbox{ as }\delta \rightarrow 0.\nonumber
\end{eqnarray} 
Hence, if $\mathcal{R}$ is defined as the norm in a Kadets-Klee space, and $\mathcal{T}_{\mathcal{U}}$ is the corresponding 
\revision{
weak
}
\revmarg{1, 4.}
topology, we even have (subsequential) norm convergence in place of $\mathcal{T}_{\mathcal{U}}$ convergence of $u_{k_*(\delta,g^\delta)}^\delta$ to an $\mathcal{R}$ minimizing solution.
\end{proposition}
\revision{
A Kadets-Klee (also called Radon-Riesz) space is a normed space in which, for any sequence $(x_n)_{n\in\N}$, convergence of the norms $\|x_n\|\to\|x\|$ and weak convergence $x_n\rightharpoonup x$ implies strong convergence $\|x_n-x\|\to0$.}
\revmarg{2, 4.}

\begin{proof}
From 
\revision{
\cite[Corollary 2.6]{ClasonKlassen18}
} 
\revmarg{1, 5.}
we conclude that for $k<k_*$, the minimizer $u_{k+1}^\delta(\rho)$ lies on the boundary of the feasible set $\mathcal{R}(u_{k+1}^\delta(\rho_k))=\rho_k$, provided $\bar{g}_k^\delta \not\in F'(u_k^\delta)B_{\rho_k}$ holds. The latter can be easily verified by noting that if there exists $ u \in B_{\rho_k}$ such that $\bar{g}_k^\delta = F'(u_k^\delta)u$, then $\|F'(u_k^\delta)(u-u_k^\delta)+F(u_k^\delta)-g^\delta\|=0$ which contradicts \eqref{choice} and \eqref{discprinc} 
\revision{due to \eqref{caseA}}.
Thus, we have
\revision{
\begin{equation}\label{Rucases}
\mathcal{R}(u_{k+1}^\delta) = \rho_k \,.
\end{equation}
}
If there exists a subsequence $\delta_l \rightarrow 0$ such that for all $l\in\N$, the iteration is already stopped at  $k=0$, we have 
\[
\|F(u_{k_*(\delta_l,g^{\delta_l})}^{\delta_l})-g\|=\|F(u_0)-g\|\leq(\tau+1)\delta_l\to0\mbox{ as }l\to0\,,
\]
hence $u_0$ solves \eqref{Fxg} and 
\[
\begin{aligned}
&\inf\{\mathcal{R}(u^*)\, : \, u^*\in B_R \mbox{ solves \eqref{Fxg}} \}\\
&=\min\{\mathcal{R}(u^*)\, : \, u^*\in B_R \mbox{ solves \eqref{Fxg}} \}=\mathcal{R}(u_0)=R^*=0.
\end{aligned}
\]
This implies trivial convergence of the subsequence to $R^*$ in this case.

Otherwise, for any subsequence $\delta_l \rightarrow 0$ and any $m\in\N$, there exists $l_m\geq m$ such  that $k_*(\delta_{l_m},g^{\delta_{l_m}})\geq1$, hence in the next to last step \eqref{caseA} holds and by \eqref{Rucases},
\[
\mathcal{R}(u^{(m)}) = \rho_{k_*(\delta_{l_m},g^{\delta_{l_m}})-1}
\]
for 
\[
u^{(m)}:=u^{\delta_{l_m}}_{k_*(\delta_{l_m},g^{\delta_{l_m}})}\,.
\]
In this case, for any $u^*$ solving \eqref{Fxg}, by \eqref{mon} we also have 
\begin{equation}\label{mon1}
\forall m\in\N\, : \quad \mathcal{R}(u^{(m)})\leq \mathcal{R}(u^*)
\end{equation}
and by Assumption \ref{ass1}, there exists a solution $u^*=u^\dagger$ such that the right hand side is finite. 
Thus, again by Assumption \ref{ass1}, $(u^{(m)})_{m\in\N}$ has a $\mathcal{T}_{\mathcal{U}}$ convergent subsequence $(u^{(m_n)})_{n\in\N}$ with limit $\tilde{u}$, which, by $\mathcal{T}_{\mathcal{U}}$ closedness of $F$ and \eqref{discprinc}, solves \eqref{Fxg}. On the other hand, from $\mathcal{T}_{\mathcal{U}}$  lower semicontinuity of $\mathcal{R}$ we conclude
\begin{equation}\label{lsc}
\mathcal{R}(\tilde{u})\leq\liminf_{n\to\infty}\mathcal{R}(u^{(m_n)})\,.
\end{equation}
This together with \eqref{mon1} implies
\[
\mathcal{R}(\tilde{u})\leq\mathcal{R}(u^*) 
\]
i.e., since $u^*$ was an arbitrary solution of \eqref{Fxg}, 
\[
\inf\{\mathcal{R}(u^*)\, : \, u^*\in B_R \mbox{ solves \eqref{Fxg}} \}=
\min\{\mathcal{R}(u^*)\, : \, u^*\in B_R \mbox{ solves \eqref{Fxg}} \}=\mathcal{R}(\tilde{u}).
\]
Using again \eqref{lsc} and \eqref{mon1} with $u^*=\tilde{u}$, we end up with 
\[
\lim_{n\to\infty}\mathcal{R}(u^{(m_n)})=\mathcal{R}(\tilde{u})=R^*\,.
\]
A subsequence-subsequence argument yields the assertion.
\end{proof}

To obtain convergence rates in the Bregman distance with respect to $\mathcal{R}$
\begin{equation*}
D_\epsilon (\tilde{u},u)=\mathcal{R}(\tilde{u})-\mathcal{R}(u)-\langle \epsilon,\tilde{u}-u\rangle
\end{equation*}
for some $\epsilon$ in the subdifferential $\partial \mathcal{R}(u)$ we make use of a variational source condition (cf., e.g., \cite{BrediesLorenz09,Flemming11,Grasmair10,HKPS07,HofmannMathe12,HohageWeidling17})
\revision{at some solution $u^*\in B_R$ of \eqref{Fxg}
\begin{equation}\label{source}
\begin{aligned}
&\exists \epsilon^\dagger \in \partial\mathcal{R}(u^*)\, \exists \beta\in [0,1) \, \forall \tilde{u} \in B_{\mathcal{R}(u^*)} \ :\\
&
-\langle \epsilon^\dagger,\tilde{u}-u^*\rangle \leq \beta D_{\epsilon^\dagger}(\tilde{u},u^*)+\phi(\|F(\tilde{u})-F(u^*)\|)
\end{aligned}
\end{equation}
}
for some index function $\phi:\mathbb{R}^+\longrightarrow\mathbb{R}^+$ (i.e., $\phi$ monotonically increasing and $\lim_{t\rightarrow 0} \phi (t)=0$),
\revision{
which we assume to be nonempty for this purpose.
}
\revmarg{1, 6.}

\begin{corollary} \label{cor:rates}
Under the assumptions of Theorem \ref{theo_conv} and the variational source condition \eqref{source}, $u^\delta_{k_*(\delta,g^\delta)}$ satisfies the convergence rate
\begin{equation}\label{rate}
D_{\epsilon^\dagger}(u^\delta_{k_*(\delta,g^\delta)},u^*)\leq \frac{1}{1-\beta}\phi((\tau +1)\delta).
\end{equation}
\end{corollary}
\begin{proof}
This is a consequence of 
\revision{
\cite[Proposition 2.9]{KK17}
}
\revmarg{1, 7.}
\revmarg{2, 5.}
and \eqref{mon} as well as \eqref{discprinc}.
\revision{
For completeness of exposition, we provide the short proof here.
    \[
        \begin{aligned}
            &D_{\epsilon^\dagger}(u^\delta_{k_*(\delta,g^\delta)},u^*)
            =\mathcal{R}(u^\delta_{k_*(\delta,g^\delta)}) - \mathcal{R}(u^*)
            -\langle \epsilon^\dagger , u^\delta_{k_*(\delta,g^\delta)}-u^*\rangle\\
            &\leq \beta D_{\epsilon^\dagger}(u^\delta_{k_*(\delta,g^\delta)},u^*)+\phi(\|F(u^\delta_{k_*(\delta,g^\delta)})-g\|)\\
            &\leq \beta D_{\epsilon^\dagger}(u^\delta_{k_*(\delta,g^\delta)},u^*)+\phi((\tau+1)\delta)
        \end{aligned}
    \]
}
\end{proof}

\section{Convergence of discretized approximations} \label{sec:discr}

We now consider a discretized version for the actual numerical solution of  \eqref{ivanov} arising from restriction of the minimization to finite dimensional subspaces $X_h^k$ containing $u_0$ and leading to discretized iterates $u_{k,h}^\delta$ and an approximate version $F_h^k$ of the forward operator:
\begin{equation}\label{ivanovD}
\begin{aligned}
u_{k+1,h}^\delta(\rho) \in \mbox{argmin}_{u\in\mathcal{D}(F) \cap X_h^k} \|{F^k_h}'(u_{k,h}^\delta)(u-u_{k,h}^\delta)+F_h^k(u_{k,h}^\delta)-g^\delta\| 
\\
\mbox{ such that } \mathcal{R}(u)\leq \rho,
\end{aligned}
\end{equation} with stopping index $k_*=k_*(\delta,g^\delta)$ according to the discretized discrepancy principle
\begin{equation}\label{discprincD}
k_*=k_*(\delta,g^\delta)=\min \{k\in\mathbb{N}_0 : \|F_h^k(u^\delta_{k,h})-g^\delta\| \leq \tau \delta\},
\end{equation} for $\tau > \frac{1+c_{tc}}{1-3c_{tc}}$.
\revision{
Here the sub- and superscripts $h$ and $k$ indicate that ${F^k_h}$ and ${F^k_h}'$ are discrete approximations of $F$ and $F'$, respectively, obtianed, e.g., by finite element discretizations on computational grids that my differ from step to step. In particular, they may be coarse at the beginning of the Newton iteration and emerge by successive mesh refinement during the iteration. Note that ${F^k_h}'$ is not necessarily the derivative of ${F^k_h}$.
}
\revmarg{2, 6.}

The regularization parameter $\rho=\rho_{k,h}$ is chosen according to the following discretized version of \eqref{choice} (relying on actually computed quantities):
\begin{equation}\label{choiceD}
\begin{aligned}\tilde{\theta}\|F_h^k(u_{k,h}^\delta)-g^\delta\|\leq \|{F^k_h}'(u_{k,h}^\delta)(u_{k+1,h}^\delta(\rho_{k,h})-u_{k,h}^\delta)+F^k_h(u_{k,h}^\delta)-g^\delta\|\\
\leq \tilde{\Theta}\|F_h^k(u_{k,h}^\delta)-g^\delta\|,
\end{aligned}
\end{equation} 
for $\tilde{\theta} <\tilde{\Theta}$ provided 
\begin{equation}\label{caseAh}
\|{F^k_h}'(u_{k,h}^\delta)(u_0-u_{k,h}^\delta)+F_h^k(u_{k,h}^\delta)-g^\delta\|\ge\tilde{\Theta}\|F_h^k(u_{k,h}^\delta)-g^\delta\|;
\end{equation}
in this case we set $u_{k+1,h}^\delta := u_{k+1,h}^\delta (\rho_{k,h})$. 
\revision{
Again we will show that \eqref{caseAh} is always satisfied.
}
\revmarg{1, 8.}
\revmarg{2, 1.}

\revision{
The tangential cone condition can usually not be expected to be transferrable from the continuous to the discretized setting, as already the simple setting of ${F^k_h}=P^k_hF$, ${F^k_h}'=P^k_hF'$ with a projection operator $P^k_h$ shows, since the right hand side $c_{tc}\|F_h^k(\bar{u}_h)-F_h^k(u_h)\|=c_{tc}\|P^k_h(F(\bar{u}_h)-F(u_h))\|$ will usually be too weak to estimate the (projected) first order Taylor remainder.
This can also be seen from the fact that the adjoint range invariance condition $F'(\bar{u})=R_u^{\bar{u}} F'(u)$ (with some bounded linear operator $R_u^{\bar{u}}$ close to the identity), that is often used to verify the tangential cone condition, does not imply its projected version $P^k_hF'(\bar{u}_h)=\tilde{R}_{u_h}^{\bar{u}_h} P^k_h F'(u_h)$.
Thus, in order to be able to employ the continuous version \eqref{ctc}, 
}
\revmarg{2, 7.}
we also define the auxiliary continuous iterates (for an illustration, see \cite[Figure 1]{KKV14I}):
\begin{equation}\label{ivanovA}
\begin{aligned}
u_{k+1,a}^\delta(\rho) \in \mbox{argmin}_{u\in\mathcal{D}(F)} \|F'(u_{k,h}^\delta)(u-u_{k,h}^\delta)+F(u_{k,h}^\delta)-g^\delta\| \\
\mbox{ such that } \mathcal{R}(u)\leq \rho,
\end{aligned}
\end{equation} and the parameter $\rho=\rho_{k,a}$ is given by
\begin{equation}\label{choiceA}
\begin{aligned}
\hat{\theta}\|F(u_{k,h}^\delta)-g^\delta\|\leq \|F'(u_{k,h}^\delta)(u_{k+1}^\delta(\rho_{k,a})-u_{k,h}^\delta)+F(u_{k,h}^\delta)-g^\delta\|\\
\leq \hat{\Theta}\|F(u_{k,h}^\delta)-g^\delta\|,
\end{aligned}
\end{equation} 
provided \eqref{caseAh} above holds.

As we will show now, this still allows to prove closeness to some projection $P_h^k x^*$  (e.g., a metric one) of an exact solution $x^*$ onto the finite dimensional space $X_h^k$, provided certain accuracy requirements are met.
\revision{
Note that also in $P_h^k$, the discretization level may depend on $k$ and will typically get finer for increasing $k$ in order to enable convergence in the sense of condition \eqref{limsupRh} below.
}
\revmarg{2, 8.}

\begin{corollary} \label{cor:discr}
Let the assumptions of Theorem \ref{theo_conv} be satisfied and assume that the discretization error estimates
\begin{eqnarray}
&&\|F(u_{k+1,h}^\delta)-g^\delta\|-\|F(u_{k+1,a}^\delta) -g^\delta\| \leq  \eta_{k+1}\label{eta}\\
&&| \|F^k_h(u_{k,h}^\delta)-g^\delta\|-\|F(u_{k,h}^\delta) -g^\delta\| | \leq \xi_{k} \label{eps}\\
&&\|{F^k_h}'(u_{k,h}^\delta)(u_0-u_{k,h}^\delta)+F^k_h(u_{k,h}^\delta)-g^\delta\|\nonumber\\
&& \qquad-\|F'(u_{k,h}^\delta)(u_0-u_{k,h}^\delta)+F(u_{k,h}^\delta)-g^\delta\| |
\leq \gamma_{k}\label{gam}\\
&&
\|{F^k_h}'(u_{k,h}^\delta)(P_h^ku^*-u_{k,h}^\delta)+F^k_h(u_{k,h}^\delta)-g^\delta\|\nonumber\\
&& \qquad
-\|F'(u_{k,h}^\delta)(u^*-u_{k,h}^\delta)+F(u_{k,h}^\delta)-g^\delta\|
\leq  \zeta_{k} \label{alp} 
\end{eqnarray} 
for any solution $u^* \in B_R$ of \eqref{Fxg} hold with 
\begin{equation}\label{tee}
\begin{aligned}
&\eta_{k+1} \leq c_\eta \|F_h^k(u_{k,h}^\delta)-g^\delta\|, \quad 
\varepsilon_k \leq c_\xi \|F_h^k(u_{k,h}^\delta)-g^\delta\|, \\
&\gamma_k \leq c_\gamma\|F_h^k(u_{k,h}^\delta)-g^\delta\|, \quad 
\zeta_k \leq c_\zeta \|F^k_h(u_{k,h}^\delta)-g^\delta\|,
\end{aligned}
\end{equation}
for all $k \leq k_*(\delta,g^\delta)$ and sufficiently small constants $c_\eta, c_\xi, c_\gamma, c_\zeta > 0$, and 
\[
\begin{aligned}
&1\geq\tilde{\Theta}\geq \tilde{\theta}\geq \left(c_{tc}+\frac{1+c_{tc}}{\tau}\right)(1+c_\xi)+c_\zeta\\
&1\geq\hat{\Theta}\geq \hat{\theta}\geq c_{tc}+\frac{1+c_{tc}}{\tau(1-c_\xi)}\\
&\hat{\Theta}\leq\frac{\tilde{\Theta}-c_\gamma}{1+c_\xi}\,, \quad
\hat{\Theta} <\left(1-\frac{c_\eta}{1-c_\xi}\right)(1-c_{tc}) -c_{tc}\\
&\tilde{\Theta}< \left(\left(1-\frac{c_\eta}{1-c_\xi}\right)(1-c_{tc}) -c_{tc}\right)(1-c_\xi)-c_\gamma
\end{aligned}
\]
\revision{
and
\begin{equation}\label{condthetatilde}
\tilde{\Theta}\leq(1-c_\xi)\left(\frac{1-c_{tc}}{\hat{q}}-c_{tc}\right)-c_\gamma
\end{equation}
}
\revmarg{1, 8.}
\revmarg{2, 1.}
Then, the assertions \revision{(i)}-(iii) of Theorem \ref{theo_conv} remain valid for $u_{k_*(\delta,g^\delta),h}^\delta$ in place of $u_{k_*(\delta,g^\delta)}^\delta$, and \eqref{discprincD} in place of \eqref{discprinc}, as well as 
\begin{equation}\label{qhat}
\hat{q}=\frac{\max\left\{\hat{\Theta}\,,\ \frac{c_\gamma+\tilde{\Theta}}{1-c_\xi}\right\}
+c_{tc}}{1-c_{tc}}+\frac{c_\eta}{1-c_\xi}
<1
\end{equation}
in place of $q$ and $\mathcal{R}(P_h^k u^*)$ in place of  $\mathcal{R}(u^*)$.

If $\mathcal{R}(P_h^k u^*)$ is uniformly bounded then this implies $\mathcal{T}_{\mathcal{U}}$-subsequential convergence of 
\revision{
$u_{k_*(\delta,g^\delta),h}^\delta$
} 
\revmarg{2, 8.}
to a solution of \eqref{Fxg} as $\delta\to0$.
\end{corollary}

\begin{proof}
First of all note that \eqref{eps}, \eqref{tee} imply 
\begin{equation}\label{estFFhk}
\begin{aligned}
&\|F(u_{k,h}^\delta)-g^\delta\|\leq(1+c_\xi) \|F_h^k(u_{k,h}^\delta)-g^\delta\|\\ \mbox{ and }
&\|F_h^k(u_{k,h}^\delta)-g^\delta\|\leq \frac{1}{1-c_\xi}\|F(u_{k,h}^\delta)-g^\delta\|\,.
\end{aligned}
\end{equation}

The existence of minimizers $u_{k+1,h}^\delta(\rho_{k,h})$ and $u_{k+1,a}^\delta(\rho_{k,a})$ of \eqref{ivanovD} and \eqref{ivanovA} for fixed $k, h, u_{k,h}^\delta, g^\delta$ and $\rho_{k,a}, \rho_{k,h}$ follows from the direct method of calculus of variations:
 
Fix $k$ such that $\|F^k_h(u_{k,h}^\delta)-g^\delta\| > \tau\delta$ and 
\revision{
assume \eqref{caseAh} to hold.
}
\revmarg{1, 8.}
\revmarg{2, 1.}
Similarly to the continuous setting, defining $d_h(\rho,g):=\min\{\|{F^k_h}'(u_{k,h}^\delta)u_h-g\|: u_h\in B_\rho\cap X_h^k\}$ and using $\bar{g}_{k,h}^\delta:= g^\delta +{F^k_h}'(u_{k,h}^\delta)u_{k,h}^\delta -F^k_h(u_{k,h}^\delta)$, we conclude 
$d_h(0,\bar{g}_{k,h}^\delta) \geq \tilde{\Theta} \|F^k_h(u_{k,h}^\delta)-g^\delta\|$, 
\revision{
by \eqref{caseAh}.
}
\revmarg{1, 8.}
\revmarg{2, 1.}
On the other hand, from the discretization error estimates \eqref{eta}--\eqref{tee}, for any solution $u^* \in B_R$ of \eqref{Fxg}, using \eqref{m1} we get
\begin{equation}\label{estdhust}
\hspace*{-0.5cm}\begin{aligned}
&d_h(\mathcal{R}(P_h^k u^*),\bar{g}_{k,h}^\delta) \leq  \|F'(u_{k,h}^\delta)(u^*-u_{k,h}^\delta)+F(u_{k,h}^\delta)-g^\delta\|+\zeta_k\\
&\leq\left(\left(c_{tc}+\frac{1+c_{tc}}{\tau}\right)(1+c_\xi)+c_\zeta\right)\|F_h^k(u_{k,h}^\delta)-g^\delta\|
\leq\tilde{\theta}\|F_h^k(u_{k,h}^\delta)-g^\delta\|\,.
\end{aligned}
\end{equation}
Hence, again by the continuity of the distance mapping $\rho \mapsto d_h(\rho,\bar{g}_{k,h}^\delta)$ for $\rho \in [0,\|\bar{g}_{k,h}^\delta\| ]$ together with the Intermediate Value Theorem, we have existence of $\rho_{k,h} \in [0,\mathcal{R}(u^*)]$ such that \eqref{choiceD} holds. 
Also, from \eqref{estdhust} and from the fact that the lower bound in \eqref{choiceD} means $d_h(\rho_{k,h},\bar{g}_{k,h}^\delta) \geq \tilde{\theta}\|F^k_h(u_{k,h}^\delta)-g^\delta\|$ we conclude $d_h(\mathcal{R}(P_h^ku^*),\bar{g}_{k,h}^\delta) \leq d_h(\rho_{k,h},\bar{g}_{k,h}^\delta)$, which, together with the monotone decrease of $\rho \rightarrow d_h(\rho,\bar{g}_{k,h}^\delta)$ implies $\mathcal{R}(P_h^ku^*)\geq \rho_{k,h}\geq\mathcal{R}(u_{k+1,h}^\delta)$.

Well-definedness of $\rho_{k,a}$ according to \eqref{choiceA} 
\revision{
under condition \eqref{caseAh}.
}
\revmarg{1, 8.}
\revmarg{2, 1.}
can as well be concluded like in the continuous case (with $u_k^\delta$ replaced by $u_{k,h}^\delta$), using the fact that by the given error estimates, for $d(\rho,g):=\min\{\|F(u_{k,h}^\delta)u-g\|: u\in B_\rho\}$ and $\bar{g}_{k}^\delta=g^\delta +F'(u_{k,h}^\delta)u_{k,h}^\delta -F(u_{k,h}^\delta)$
\[
\begin{aligned}
&d(0,\bar{g}_{k}^\delta) =\|F'(u_{k,h}^\delta)(u_0-u_{k,h}^\delta)+F(u_{k,h}^\delta)-g^\delta\|
\geq (\tilde{\Theta}-c_\gamma) \|F^k_h(u_{k,h}^\delta)-g^\delta\|\\
&\geq \frac{\tilde{\Theta}-c_\gamma}{1+c_\xi} \|F(u_{k,h}^\delta)-g^\delta\|
\geq \hat{\Theta} \|F(u_{k,h}^\delta)-g^\delta\|
\end{aligned}
\]
\revision{
due to \eqref{caseAh}.
}
\revmarg{1, 8.}
\revmarg{2, 1.}
and 
\[
\begin{aligned}
&d(\mathcal{R}(u^*),\bar{g}_{k}^\delta)\leq \|F'(u_{k,h}^\delta)(u^*-u_{k,h}^\delta)+F(u_{k,h}^\delta)-g^\delta\|\\
&\leq c_{tc}\|F(u_{k,h}^\delta)-F(u^*)\|+\delta\leq\left(c_{tc}+\frac{1+c_{tc}}{\tau(1-c_\xi)}\right)\|F(u_{k,h}^\delta)-g^\delta\|\\
&\leq \hat{\theta}\|F(u_{k,h}^\delta)-g^\delta\|\,.
\end{aligned}
\]

In order to obtain geometric decay of the residuals like in \eqref{contr} for the discretized version, note that for the auxiliary sequence defined by \eqref{ivanovA}, \eqref{choiceA}, we have by \eqref{ctc} and the discretization error estimates,
\begin{equation}\label{est_qhat}
\begin{aligned}
&\hat{\Theta} \|F(u_{k,h}^\delta)-g^\delta\|\geq \|F'(u_{k,h}^\delta)(u_{k+1,a}^\delta-u_{k,h}^\delta)+F(u_{k,h}^\delta)-g^\delta\|\\
&\geq(1-c_{tc})\|F(u_{k+1,a}^\delta)-g^\delta\|-c_{tc}\|F(u_{k,h}^\delta)-g^\delta\|\\ 
&\geq(1-c_{tc})\left(\|F(u_{k+1,h}^\delta)-g^\delta\|-\frac{c_\eta}{1-c_\xi}\|F(u_{k,h}^\delta)-g^\delta\|\right)\\&-c_{tc}\|F(u_{k,h}^\delta)-g^\delta\|\,,
\end{aligned}
\end{equation}
thus we get linear decay
\begin{equation}\label{contr_qhat}
\|F(u_{k+1,h}^\delta)-g^\delta\| \leq \hat{q} \|F(u_{k,h}^\delta) -g^\delta\|,
\end{equation}
with $\hat{q}$ as in \eqref{qhat}.

\revision{
To see that \eqref{caseAh} remains valid in the next step of the discretized version, provided, additionally \eqref{condthetatilde} holds, note that in the proof of the respective part of Theorem \ref{theo_conv}, we just have to replace $u_k^\delta$ by $u_{k,h}^\delta$, $F$ by $F_h^k$, $q$ by $\hat{q}$, and $\Theta$ by $(1-c_\xi)\hat{\Theta}-c_\gamma$, and \eqref{caseAkp1} by  
\[
\begin{aligned}
&\|{F_h^{k+1}}'(u_{{k+1},h}^\delta)(u_0-u_{{k+1},h}^\delta)+F_h^{k+1}(u_{{k+1},h}^\delta)-g^\delta\|\\
&\geq\|F'(u_{{k+1},h}^\delta)(u_0-u_{{k+1},h}^\delta)+F(u_{{k+1},h}^\delta)-g^\delta\|-c_\gamma\|F_h^{k+1}(u_{{k+1},h}^\delta)-g^\delta\|\\
&\geq \left(\frac{1-c_{tc}}{\hat{q}}-c_{tc}\right)\|F(u_{{k+1},h}^\delta)-g^\delta\|-c_\gamma\|F_h^{k+1}(u_{{k+1},h}^\delta)-g^\delta\|\\
&\geq \left((1-c_\xi)\left(\frac{1-c_{tc}}{\hat{q}}-c_{tc}\right)-c_\gamma\right)\|F_h^{k+1}(u_{{k+1},h}^\delta)-g^\delta\|
\ \geq \tilde{\Theta} \|F_h^{k+1}(u_{{k+1},h}^\delta)-g^\delta\|\,.
\end{aligned} 
\]
}
\revmarg{1, 8.}
\revmarg{2, 1.}

By means of \eqref{estFFhk}, we conclude from \eqref{contr_qhat}
\begin{equation}\label{10}
(1-c_\xi)\|F^k_h(u_{k,h}^\delta)-g^\delta\| \leq \hat{q}^k \|F(u_0)-g^\delta\|.
\end{equation}
Therefore, we have the following estimate showing that the stopping index defined by \eqref{discprincD} is finite, i.e., $\|F^k_h(u_{k,h}^\delta)-g^\delta\|$ falls below $\tau\delta$ as soon as
\begin{equation}
k \geq (\log 1/\hat{q})^{-1}(\log\|F(u_0)-g^\delta\|-\log(1-c_\xi)-\log\tau -\log\delta) =: \bar{k}_*(\delta) \geq k_*(\delta,g^\delta).
\end{equation}

\end{proof}

Also Proposition \ref{prop:convR} carries over to the discretized setting as follows.

\begin{proposition} \label{prop:convRh}
Under the assumptions of Corollary \ref{cor:discr}, assuming additionally 
\begin{equation}\label{limsupRh}
\limsup_{\delta\to0}\mathcal{R}(P_h^{k_*(\delta,g^\delta)-1} u^*)\leq \mathcal{R}(u^*)
\end{equation}
for any solution $u^*\in B_R$ of \eqref{Fxg},
with $\mathcal{R}$ defined by a norm \eqref{Rnorm}, we have, for $u_{k_*{(\delta,g^\delta)_h}}^\delta$ defined by the discretized version of Algorithm \ref{IvanovAlg} (replacing \eqref{ivanov}, \eqref{discprinc}, \eqref{choice} by \eqref{ivanovD}, \eqref{discprincD}, \eqref{choiceD}), that 
\begin{eqnarray}\label{convRh}
\mathcal{R}(u_{k_*{(\delta,g^\delta)},h}^\delta)&\rightarrow &
\inf\{\mathcal{R}(u^*)\, : \, u^*\in B_R \mbox{ solves \eqref{Fxg}} \}\\
&=&\min\{\mathcal{R}(u^*)\, : \, u^*\in B_R \mbox{ solves \eqref{Fxg}} \}=:R^*
\mbox{ as }\delta \rightarrow 0.\nonumber
\end{eqnarray} 

Hence, if $\mathcal{R}$ is defined as the norm in a Kadets-Klee space, and $\mathcal{T}_{\mathcal{U}}$ is the corresponding norm topology, we even have (subsequential) norm convergence in place of $\mathcal{T}_{\mathcal{U}}$ convergence of $u_{k_*(\delta,g^\delta),h}^\delta$ to an $\mathcal{R}$ minimizing solution.
\end{proposition}

\begin{proof}
Like in the proof of Proposition \ref{prop:convR}, from \cite[Corollary 2.8]{ClasonKlassen18} we conclude that for $k<k_*$, the minimizers $u_{k+1,h}^\delta(\rho_{k,h})$, $u_{k+1,a}^\delta(\rho_{k,a})$ lie on the boundary of the feasible sets
\revision{
\[
\mathcal{R}(u_{k+1,h}^\delta) = \rho_{k,h} \,, \quad 
\mathcal{R}(u_{k+1,a}^\delta) = \rho_{k,a} 
\]
provided \eqref{caseAh} holds.
}
\revmarg{1, 8.}
\revmarg{2, 1.}

If there exists a subsequence $\delta_l \rightarrow 0$ such that for all $l\in\N$, the iteration is already stopped at  $k=0$, we have 
\[
\|F(u_{k_*(\delta_l,g^{\delta_l}),h}^{\delta_l})-g\|=\|F(u_0)-g\|\leq((1+c_\xi)\tau+1)\delta_l\to0\mbox{ as }l\to0\,,
\]
hence $u_0$ solves \eqref{Fxg} and as in the proof of Proposition \ref{prop:convR} we have trivial convergence of the subsequence to $R^*$ in this case.

Otherwise, for any subsequence $\delta_l \rightarrow 0$ there exists a subsequence $\delta_{l_m}$ such that at $k_*(\delta_{l_m},g^{\delta_{l_m}})-1$ 
\revision{
condition \eqref{caseAh} holds 
} 
\revmarg{1, 8.}
\revmarg{2, 1.}
and thus both the discrete and the continuous versions of the $k_*(\delta_{l_m},g^{\delta_{l_m}})$th iterate lie at the boundary of their respective feasible set.
Like in the continuous setting, from 
\[
\rho_{k_*(\delta_{l_m},g^{\delta_{l_m}})-1}=\mathcal{R}(u_h^{(m)})\leq \mathcal{R}(P_h^{k_*(\delta_{l_m},g^{\delta_{l_m}})-1}u^*)
\]
and \eqref{limsupRh} we obtain existence of a $\mathcal{T}_{\mathcal{U}}$ convergent subsequence of 
\[
u_h^{(m)}:=u^{\delta_{l_m}}_{k_*(\delta_{l_m},g^{\delta_{l_m}}),h}
\]
whose limit $\tilde{u}$ by $\mathcal{T}_{\mathcal{U}}$ lower semicontinuity of $\mathcal{R}$ satisfies
\begin{equation}\label{mon1h}
\mathcal{R}(\tilde{u})\leq\liminf_{n\to\infty}\mathcal{R}(u^{(m_n)})\leq\mathcal{R}(u^*)
\end{equation}
for any solution $u^*\in B_R$ of \eqref{Fxg}.
On the other hand, due to the estimate
\[
\|F(u_h^{(m)})-g\|\leq((1+c_\xi)\tau+1)\delta_{l_m}\to0\mbox{ as }m\to0\,,
\]
and by $\mathcal{T}_{\mathcal{U}}$ closedness of $F$, we have that $\tilde{u}$ itself solves \eqref{Fxg}. By \eqref{mon1h}, it is therefore a solution with minimal $\mathcal{R}$ value $R^*$.

As before, a subsequence-subsequence argument yields the assertion.
\end{proof}

Finally, we recover the convergence rates result Corollary \ref{cor:rates} under a source condition \eqref{source} and additional accuracy requirements.

\begin{corollary} \label{cor:ratesh} 
Under the assumptions of Corollary \ref{cor:discr}, assuming additionally \eqref{condthetatilde} 
and 
\[
\mathcal{R}(P_h^{k^*(\delta)-1} u^*)\leq \mathcal{R}(u^*)+C\phi((1+c_\xi)(\tau +1)\delta)
\]
and the variational source condition \eqref{source}, $u^\delta_{k_*(\delta,g^\delta),h}$ satisfies the convergence rate
\[
D_{\epsilon^\dagger}(u^\delta_{k_*(\delta,g^\delta),h},u^*)\leq (\frac{1}{1-\beta}+C)\phi((1+c_\xi)(\tau +1)\delta).
\]
\end{corollary}
\begin{proof}
Note that 
\revision{
\cite[Proposition 2.9]{KK17}
}
\revmarg{1, 7.}
\revmarg{2, 5.}
remains valid with the condition $\mathcal{R}(\hat{u}^\delta)\leq \mathcal{R}(u^*)$ replaced by
$\mathcal{R}(\hat{u}^\delta)\leq \mathcal{R}(u^*)+C\phi((1+c_\xi)(\tau +1)\delta)$.
Together with the fact that 
\[
\|F(u^\delta_{k_*(\delta,g^\delta),h})-g^\delta\|\leq (1+c_\xi)(\tau +1)\delta
\]
this yields the assertion.
\end{proof}

\section{Numerical experiments} \label{sec:numexp}

To describe the numerical implementation, we will refer to the formulation of the inverse problem as a system of model and observation equation
\begin{eqnarray}\label{Auy}
A(u,y) &=& 0;\\\label{Cyg}
C(y) &=& g.
\end{eqnarray} 
Here, $A: \mathcal{U} \times V \rightarrow W^*$ and $C: V \rightarrow \mathcal{Y}$ are the model and observation operator so that with the parameter-to-state map $S:\mathcal{U} \rightarrow V$ satisfying $A(u,S(u))=0$ and $F=C\circ S$, \eqref{Fxg} is equivalent to the all-at-once formulation \eqref{Auy}, \eqref{Cyg}. 
\revision{
In many applications, \eqref{Auy} will be given by a (partial) differential equation with boundary and/or initial conditions, hence the method considered in this paper will involve PDE constrained optimization as follows.
}
\revmarg{3, 1.}
An IRGN-Ivanov step requires solution of the minimization problem
\begin{eqnarray}\label{minAC} 
(u_{k+1,h}^\delta(\rho_{k,h}),v_{k,h}^\delta,y_{k,h}^\delta) && \hspace*{-0.5cm}
\in \mbox{argmin}_{(u,v,\tilde{y})\in \mathcal{D}(F)\times V^2} \frac{1}{2} \|C'(\tilde{y})v+C(\tilde{y})-g^\delta\|^2 \nonumber\\
\mbox{s.t. } && \hspace*{-0.5cm}
\mathcal{R}(u) \leq \rho,\\ \nonumber 
\mbox{and } \forall w \in W:&& \hspace*{-0.5cm}
\langle A_u'(u_{k,h}^\delta,\tilde{y})(u-u_{k,h}^\delta)+A_y'(u_{k,h}^\delta,\tilde{y})v,w\rangle_{W^*,W} = 0,\\ \nonumber
&& \hspace*{-0.5cm}
\langle A(u_{k,h}^\delta,\tilde{y}),w\rangle_{W^*,W} = 0,
\end{eqnarray} 
where $\rho$ is chosen according to 
\begin{equation}\label{choiceCC}
\tilde{\theta}\|C(\tilde{y})-g^\delta\| \leq \|C'(\tilde{y})v+C(\tilde{y})-g^\delta\| \leq \tilde{\Theta}\|C(\tilde{y})-g^\delta\|
\end{equation} 
provided that \begin{equation}
\|C'(\tilde{y})(y_0 - \tilde{y})+C(\tilde{y})-g^\delta\| \geq \tilde{\Theta}\|C(\tilde{y})-g^\delta\|\,,
\end{equation}   
where $y_0$ solves $A(u_0,y_0)=0$.
We assume that $C, \mathcal{R}$ and the norms are evaluated without discretization error.

In \cite[Equations (63)-(67)]{KaltenbacherPreviatti18} we provide explicitly the optimality system for \eqref{minAC}.
The strategy to find a stationary point cf. \cite{KaltenbacherPreviatti18} is to first solve
\begin{equation}
\langle A(u_{k,h}^\delta, \tilde{y}),w\rangle_{W^*,W} = 0, \forall w \in W,
\end{equation}
and then interpret the remaining system as optimality system for the following problem
\begin{eqnarray*} \nonumber
(u_{k+1,h}^\delta(\rho_{k,h}),v_{k,h}^\delta) && \hspace*{-0.5cm}
\in \mbox{argmin}_{(u,v)\in \mathcal{D}(F)\times V} \frac{1}{2} \|C'(\tilde{y})v+C(\tilde{y})-g^\delta\|^2\\ \nonumber
\mbox{s.t. } && \hspace*{-0.5cm}
\mathcal{R}(u) \leq \rho,\\ \nonumber 
\mbox{and } \forall w \in W:&& \hspace*{-0.5cm}
\langle A_u'(u_{k,h}^\delta,\tilde{y})(u-u_{k,h}^\delta)+A_y'(u_{k,h}^\delta,\tilde{y})v,w\rangle_{W^*,W} = 0.
\end{eqnarray*}

\medskip

As a model example we consider the following inverse source problem for a semilinear elliptic PDE, where the model and observation equations are given by
\begin{eqnarray*}
-\Delta y +\kappa y^3 &=&\chi_{\omega_c}u, \mbox{in } \Omega \subset \mathbb{R}^d,\\
y&=&0, \mbox{on } \partial \Omega,\\
C(y) &=& y\mid_{\omega_o} ,  \quad \|y-g^\delta\|_{L^2(\omega_o)} \leq \delta.
\end{eqnarray*}
Here we assume that $\Omega \subseteq \mathbb{R}^2$ is bounded and polygonal with boundary $\Gamma : = \partial \Omega$ and $\kappa\in\mathbb{R}$ is a parameter that allows to tune the nonlinearity of the model.
\revision{
The sets $\omega_c$ and $\omega_o$ are measurable subsets of $\Omega$
(so that the linear operators $\chi_{\omega_c}$ and $C$ are bounded)
on which the source is supported and on which measurements are taken, respectively.
}
\revmarg{2, 9.}

We intend to regularize by imposing $L^\infty$ bounds and thus use the function space setting 
\[
\begin{aligned}
&A: L^\infty(\omega_c)\times H^{1}_0 (\Omega) \to H^{-1} (\Omega)\,, \quad A(u,y) = - \Delta y + \kappa y^3 -u\,,\\
&C: H^1_0(\Omega)\longrightarrow L^2(\omega_o) \mbox{ defined by embedding into $L^2(\Omega)$ and restriction to $\omega_o$}\\
&\chi_{\omega_c}: L^\infty(\omega_c)\to L^\infty(\Omega)\,, \quad \chi_{\omega_c}u=\begin{cases}u\mbox{ on }\omega_c \\0\mbox{ on } \Omega\setminus\omega_c\end{cases}\\
&\mathcal{U}=L^\infty(\omega_c)\,, \mathcal{Y}=L^2(\omega_o)\,.
\end{aligned}
\] 
This problem has been proven to 
\revision{satisfy}
\revmarg{2, 12.}
the tangential cone condition \eqref{ctc}, cf. \cite[Section 4]{KaltenbacherPreviatti18}.

The IRGNM-Ivanov minimization step with the regularization functional $\mathcal{R}(u)=\|u\|_{L^\infty(\omega_c)}$ is given by (skipping the discretization subscript $h$ in the notation)
\begin{eqnarray*}
(u_{k+1}^\delta,v_k^\delta,y_k^\delta) && \hspace*{-0.5cm}
\in \mbox{argmin}_{(u,v,\tilde{y})\in L^\infty(\omega_c)\times (H^1_0(\Omega))^2} \frac{1}{2} \|v+\tilde{y}-g^\delta\|^2_{L^2(\omega_o)},\\
\mbox{ s.t. } && \hspace*{-0.5cm}
 -\rho_k \leq u(x)\leq \rho_k, \mbox{ a.e. in } \Omega, \\
\mbox{ and } \forall w\in H^1_0 (\Omega): && \hspace*{-0.5cm}
\int_{\Omega}(\nabla v \nabla w+3\kappa \tilde{y}^2 vw)dx = \int_{\Omega} w(u-u_k^\delta)dx,\\
&& \hspace*{-0.5cm}
\int_{\Omega} (\nabla \tilde{y}\nabla w+\kappa \tilde{y}^3w)dx= \int_{\Omega} wu_k^\delta dx .
\end{eqnarray*}
For the Gauss-Newton step, one needs to first solve the equation
\begin{equation}\label{nleq}
-\Delta \tilde{y} +\kappa \tilde{y}^3 = u_k^\delta
\end{equation} 
and set $y_k^\delta=\tilde{y}$ -- note that \eqref{nleq} is actually decoupled from the minimization problem -- and then, solve the following optimality system with respect to $(u,v,p)$ (written in a strong formulation)
\begin{eqnarray*}
\|u\|_{L^{\infty}(\omega_c)} \leq \rho_k  \mbox{ and } \int_\Omega (u^{*} -u)p dx & \leq& 0, \forall u^{*}\in B_{\rho_k}^{L^{\infty}(\omega_c)}\\
-\Delta p +3\kappa (y_k^\delta)^2p +2v +2y_k^\delta&=&2g^\delta \\
-\Delta v +3\kappa (y_k^\delta)^2v -u&=&-u^\delta_k,
\end{eqnarray*} 
which can be interpreted as the optimality system for the minimization problem
\begin{eqnarray}\label{linearivanovn}
(u_{k+1}^\delta,v_{k}^\delta)&&\in\mbox{argmin}_{(u,v)\in L^{\infty}(\omega_c)\times H^1_0(\Omega)} \frac{1}{2}\|y_k^\delta+v-g^\delta\|^2_{L^2(\omega_o)},\\ \nonumber
\mbox{ s.t. } && -\rho_k \leq u(x)\leq \rho_k, \mbox{ a.e. in } \Omega, \\ \nonumber
 && -\Delta v +3\kappa (y_k^\delta)^2 v = u-u^\delta_k, \\ \nonumber
\end{eqnarray} 
where $\rho_k$ is computed according to
\begin{equation}\label{choicepm}
\tilde{\theta}\|y_k^\delta - g^\delta \|_{L^2(\Omega)} \leq \|v + y_k^\delta-g^\delta\|_{L^2(\Omega)} \leq \tilde{\Theta}\|y_k^\delta - g^\delta \|_{L^2(\Omega)}.
\end{equation}

In order to solve \eqref{nleq} numerically, as in \cite{KaltenbacherPreviatti18}, we apply a damped Newton iteration to the equation $\Phi(\tilde{y}) = 0$ where
\[
\Phi: H_0^1 (\Omega) \rightarrow H^{-1}(\Omega), \quad \Phi(\tilde{y}) = -\Delta \tilde{y} +\kappa \tilde{y}^3 - u_k^\delta,
\]
which leads to the iteration
\[ \tilde{y}^{l+1}=\tilde{y}^l - (-\Delta \tilde{y} +3\kappa (\tilde{y}^l)^2)^{-1}(-\Delta \tilde{y} +\kappa (\tilde{y}^l)^3 - u_k^\delta)
\]
which is stopped as soon as the residual $\|\Phi(\tilde{y}^l)\|_{H^{-1}(\Omega)}$ has been reduced to a certain tolerance $tol$ (which we set to $1.e-6$ in our computations), see Algorithm \ref{algo:dampedNewton}. 
\begin{algorithm}
\caption{Solving Nonlinear Equation: Damped Newton\label{algo:dampedNewton}}
\begin{algorithmic}[1]
\State \textbf{Input:} $u_k^\delta, g^\delta$.
\State Choose initial guess $\tilde{y}^0$ (e.g., $\tilde{y}^0 \equiv 0$) and compute $\|\Phi(\tilde{y}^0)\|_{H^{-1}(\Omega)}$.
\While{$\|\Phi(\tilde{y}^l)\|_{H^{-1}(\Omega)} > tol$}
\State Compute $d$ by solving $-\Delta d +3\kappa (\tilde{y}^l)^2d =\Delta \tilde{y}^l -\kappa (\tilde{y}^l)^3 +u_k^\delta$.
\State Set $s = 1$.
\State Compute $\|\Phi(\tilde{y}^l+sd)\|_{H^{-1}(\Omega)}$.
\While{$\|\Phi(\tilde{y}^l+sd)\|_{H^{-1}(\Omega)} > 0.8 \|\Phi(\tilde{y}^l)\|_{H^{-1}(\Omega)}$}
\State Set $s = 0.5s$.
\State Compute $\|\Phi(\tilde{y}^l+sd)\|_{H^{-1}(\Omega)}$.
\EndWhile
\State $\tilde{y}^{l+1} = \tilde{y}^l +sd$.
\EndWhile
\Return $\tilde{y}$.
\end{algorithmic}
\end{algorithm}
For solving \eqref{linearivanovn} numerically  we apply a semi-smooth Newton method together with a Moreau-Yosida regularization in order to make the problem more convex, see Algorithm \ref{algo:semismoothNewton}. Then, we perform a search for the regularization parameter based on \eqref{choicepm}, see Algorithm \ref{algo:bisection}.
We point out that this search for the regularization parameter is similar to Algorithm 2 in \cite{ClasonKlassen18}, while here we use multiples of the discrepancy instead of $\delta$ and $\tau \delta$. Another difference to \cite{ClasonKlassen18} is the fact that a damped semi-smooth Newton method is performed there to solve the linear discretized (hence semismooth) problem and here, we apply a semi smooth Newton method with Moreau-Yosida regularization (to guarantee semismoothness) to solve our linearized problem.
Using Moreau-Yosida regularization, the problem we are actually solving in place of \eqref{linearivanovn} is
\begin{eqnarray*}
(u_{k+1}^\delta,v_k^\delta) &&\hspace*{-0.5cm}
\in \mbox{argmin}_{u,v\in L^\infty(\Omega)\times H^1_0(\Omega)} \frac{1}{2}\|v+y_k^\delta-g^\delta\|^2_{L^2(\Omega)} + \frac{\gamma}{2}\|u\|_{L^2(\Omega)}^2\\
\mbox{s.t. } && \hspace*{-0.5cm}
-\rho_k \leq u(x) \leq \rho_k, \mbox{ a.e. in } \Omega,\\
&&\hspace*{-0.5cm}
\left(-\Delta  +3\kappa (y_k^\delta)^2\right)v = u-u_k^\delta,\\
\end{eqnarray*}
which enhances convexity and thus makes the numerical solution more stable. 
We point out that regularization still relies on the $L^\infty$ bounds imposed by Ivanov regularization only, since the parameter $\gamma$ is chosen very small ($\gamma=1.e-9$ in our computations).

For this problem, we have the following optimality conditions
\begin{eqnarray}
\left(-\Delta+3\kappa(y_k^\delta)^2\right)v -u &=& -u_k^\delta,\label{state}\\
v + \left(-\Delta+3\kappa(y_k^\delta)^2\right)^*p &=& -y_k^\delta +g^\delta,\label{adjoint}\\
u - \mbox{proj}_{[-\gamma \rho, \gamma \rho]}(\frac{1}{\gamma}p) &=& 0,\label{gradient}
\end{eqnarray}
where $p\in H_0^1(\Omega)$ is the adjoint state and $\mbox{proj}_{[-\gamma \rho,\gamma \rho]}: L^2(\Omega) \rightarrow L^2(\Omega)$,
\begin{equation}\label{proj}
[\mbox{proj}_{[-\gamma \rho, \gamma \rho]}(p)](x) =: \begin{cases} \gamma \rho, \mbox{ if } p(x) > \gamma \rho,\\ p(x), \mbox{ if } p(x) \in [-\gamma \rho, \gamma \rho],\\ -\gamma \rho, \mbox{ if } p(x) < - \gamma \rho. \end{cases}
\end{equation}
The optimality system \eqref{state}--\eqref{gradient} is an operator equation for the unknowns $(u,v,p)$ to which we apply a semismooth Newton method, which after dicretization by piecewise linear elements for $(u,v,p)$ leads to the linear system $Hx=b$ with 
\begin{equation}\label{sysKMI}
H=\left( \begin{array}{ccc} K&0&-M\\M&K^*&0\\0&-\frac{1}{\gamma}As&I\\ \end{array} \right), \quad 
x=\left( \begin{array}{c} v_{k}^\delta\\p_{k}^\delta\\u_{k+1}^\delta \end{array} \right), \quad 
b=\left( \begin{array}{c} -Mu_k^\delta\\ M(-y_k^\delta +g^\delta)\\ \rho (1_{\mathcal{A}_+}-1_{\mathcal{A}_-})\end{array}\right),
\end{equation}
where $K$, $K^*$ are the stiffness matrices of the primal and adjoint problems \eqref{state}, \eqref{adjoint}, respectively, $M$ is the mass matrix, $I$ is the identity matrix, 
$\mathcal{A}_+,\mathcal{A}_-,\mathcal{I}$ are the active and inactive sets defined as
\begin{equation*}
\mathcal{A}_+ : =  \{i : p_i > \rho \gamma \}, \quad \mathcal{A}_- := \{i : p_i < - \rho \gamma \}, \quad\mathcal{I} := \{i : |p_i| \leq \rho\gamma \},
\end{equation*}
the corresponding characteristic functions are defined by $[1_{\mathcal{A}_-}]_i = 1$, if $i\in \mathcal{A}_-$ and $0$ else, (and analogously for $\mathcal{A}_+$, $\mathcal{I}$). The matrix $As$ is a diagonal matrix defined as $I - \mbox{diag}(1_{\mathcal{A}_+}+1_{\mathcal{A}_-})$,
$u_{k+1}^\delta$ is the iterate for the searched for source, $v_k$ the linearized state and $p_k$ the Lagrange multiplier for the PDE constraint. 
Note that $u_k^\delta$ is obtained from the previous Gauss Newton step, $y_k^\delta$ is precomputed by solving \eqref{nleq} with a damped Newton method, and $g^\delta$ is the noisy data.

Here we have made a transition from function space notation to coefficient vectors with respect to basis functions. More precisely, we consider a finite element space $Y_h \subset H_0^1(\Omega)$, spanned by the usual continuous piecewise linear nodal basis functions based on the vertices $(x_i)_{i=1}^n$ of a regular triangulation of $\bar{\Omega}$. Notationally, we identify $u_{k+1}^\delta,v_k,p_k \in Y_h$ with their respective coefficient vectors for this nodal basis, these actually contain the values at the nodes in this case. The fact that functions in $Y_h$  attain their maximum and minimum at the nodes, allows a straightforward definition of active and inactives sets as well as projections \eqref{proj}. Also, now $g^\delta$ denotes the nodal basis coeffficient vector of the $L^2$ projection of the continuous data onto $Y_h$. 

We here used the locally superlinearly convergent semismooth Newton method \cite{HintermullerItoKunisch02}, \cite{Ulbrich11}. It is known that the pointwise projection \eqref{proj} is semismooth from $L^p(\Omega)\rightarrow L^q(\Omega)$ if and only if $p> q$, and hence the last equation of our optimality conditions would not be semismooth with respect to $u$ unless we had addressed this issue by adding a small Tikhonov (or Moreau-Yosida) term. We emphasize once more that this term is only used for the purpose of numerical efficiency, and, since it goes with a very small factor ($\gamma=1.e-9$ in our computations) does not interfere with the actual regularization, relying on Ivanov regularization only.

The algorithm (cf. Algorithm \ref{algo:semismoothNewton} below) gradually drives $\gamma$ to zero along a geometrically decaying sequence and for each fixed $\gamma$, solves the linear system \eqref{sysKMI} and updates the active sets until they do not change any more; we set an upper bound $i_{\max}$ (thirty in our computations) on the number of Newton iterations.  Then we decrease $\gamma$ and continue the process, using the previously found point as an initial guess for the semismooth Newton iteration on this $\gamma$ level, expecting it to lie in the region of convergence of the method. Thus we perform a numerical continuation method along the parameter $\gamma$.

We say that the algorithm converged succesfully if in addition to having no updates in the active sets, we also have achieved $\gamma$ at the order of $1e-9$ and the norm of the gradient, i.e., the residual of the optimality system is below $1e-9$, i.e., 
\begin{equation}\label{grad}
\| grad \| : = \left|\left| \left( \begin{array}{c} Kv_k -M(u_{k+1}-u_k^\delta) \\ Mv_k +K^*p_k -M(y_k^\delta-g^\delta) \\ u_{k+1} - \mbox{proj}_{[-\gamma \rho, \gamma \rho]} (\frac{1}{\gamma}p) \end{array}\right)\right|\right| < 1e-9.
\end{equation} 
In this case, a bolean variable \textit{conv} $=$ True is returned; it returns False otherwise. 

\begin{algorithm}
\caption{Solving Linearized Ivanov Problem: Semi-smooth Newton with Moreau-Yosida Regularization\label{algo:semismoothNewton}}
\begin{algorithmic}[1]
\State \textbf{Input:} $\rho, u_k^\delta, \tilde{y}, g^\delta$, $\gamma_0$
\For{$\ell=0,1,2,\ldots$}
\State Set $\gamma=\gamma_0 2^{-\ell}$
\For{$i = 0,\dots, i_{\max}$}
\State Update $\mathcal{A}_+, \mathcal{A}_-$.
\State Compute $K, K^*, M, As$.
\State Compute $H$, $b$, according to \eqref{sysKMI}
\State Compute $x_{k+1}=(v_k^\delta, p_k^\delta, u_{k+1}^\delta)$ by solving $Hx_{k+1}=b$
\If{no changes in $\mathcal{A}_+, \mathcal{A}_-$}
\State \textbf{break.}
\EndIf
\EndFor
\State Compute $\|grad\|$, according to \eqref{grad}
\If{$i = 0$ and $\gamma < 1e-9$ and $\|grad\| \leq  1e-9$}
\State \textit{conv} = \textit{True}.
\State \textbf{break}.
\EndIf
\EndFor
\Return $u, v$, \textit{conv}
\end{algorithmic}
\end{algorithm}


The search for the regularization radius is described in Algorithm \ref{algo:bisection}, it basically relies on a bisection method. 
By our existence proof of a regularization parameter $\rho_k$, both phases of Algorithm \ref{algo:bisection} terminate after finitely many steps.
\begin{algorithm}
\caption{Search for Regularization Radius\label{algo:bisection}}
\begin{algorithmic}[1]
\State \textbf{Input:} $\rho_{start}, \tilde{\theta}, \tilde{\Theta}, u_k^\delta, \tilde{y}$, $g^\delta$.
\State Phase I (enlargement of search interval)
\State Set $\rho = \rho_{start}$
\State Set $\bar{\theta} = \frac{1}{2}(\tilde{\theta}+\tilde{\Theta})\|\tilde{y}-g^\delta\|$.
\For{$i = 1,2,\dots$}
\State Compute $u$, $v$ by Algorithm \ref{algo:semismoothNewton} with input $\rho, u_k^\delta, \tilde{y}, g^\delta$
\State Determine $d(\rho)=\|\tilde{y} + v - g^\delta\|$.
\If{$\tilde{\theta} \leq d(\rho) \leq \tilde{\Theta}$ and \textit{conv} = \textit{True}}
\Return $u_{k+1}^\delta:=u$ and $\rho$.
\EndIf
\If{$d(\rho) < \tilde{\theta}$ and \textit{conv} = \textit{True}}
\State \textbf{break.}
\EndIf
\State Set $\rho = \rho + \rho_{start}$.
\State $i=i+1$.
\EndFor
\State Phase II (bisection for fine search)
\State Set $i=0$.
\State Set $a=0,b=\rho$.
\State Compute $u$, $v$ by Algorithm \ref{algo:semismoothNewton} with input $\rho=a, u_k^\delta, \tilde{y}, g^\delta$
\For{$i = 1,2,\dots$}
\State Compute $u$, $v$ by Algorithm \ref{algo:semismoothNewton} with input $\rho, u_k^\delta, \tilde{y}, g^\delta$
\State Determine $d(\rho)=\|\tilde{y} + v-g^\delta\|$.
\If{$\tilde{\theta} \leq d(\rho) \leq \tilde{\Theta}$ and \textit{conv} = \textit{True}}
\Return $u_{k+1}^\delta:=u$ and $\rho$.
\EndIf
\If{$sign(d(\rho)-\bar{\theta})=sign(d(a)-\bar{\theta})$}
\State Set $a=\rho$.
\Else
\State Set $b=\rho$.
\EndIf
\State Set $\rho=\frac{1}{2}(a+b)$.
\State $i=i+1$.
\EndFor
\end{algorithmic}
\end{algorithm}

\begin{figure}[h]
\begin{tabular}{cc}
\includegraphics[width=0.43\textwidth]{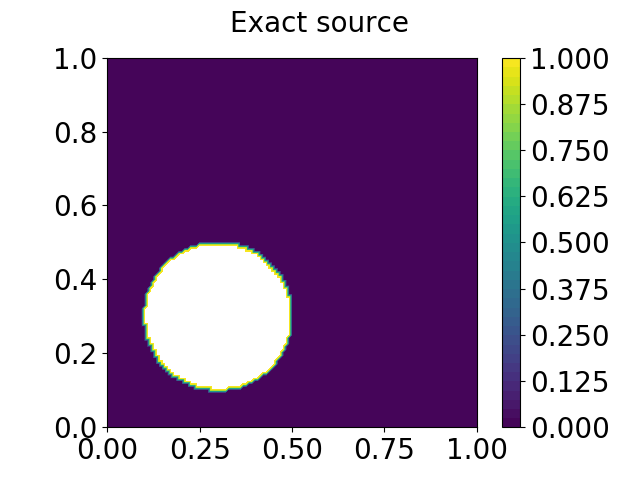} & 
\includegraphics[width=0.34\textwidth]{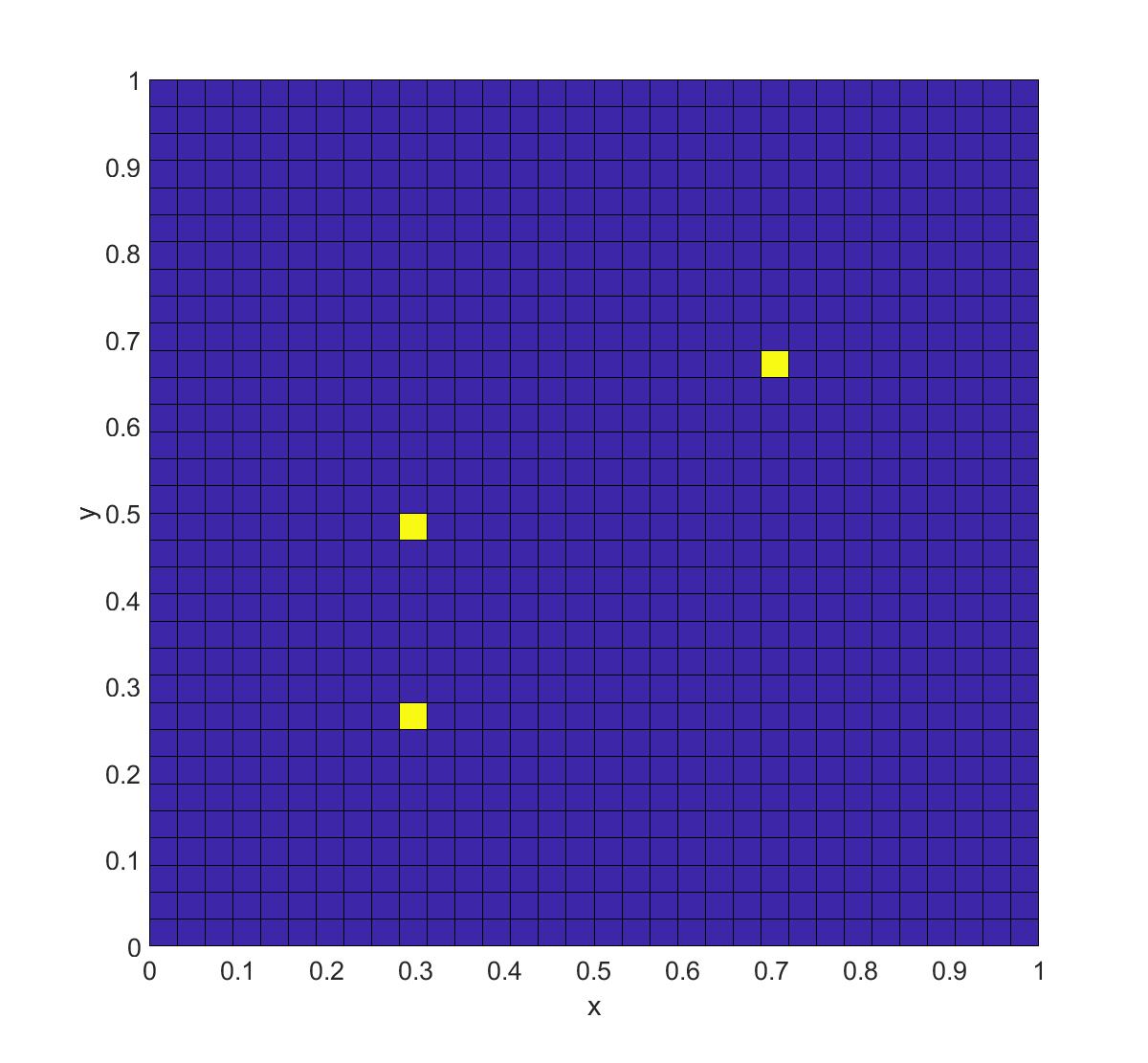} \\
\end{tabular}
\caption{Exact solution $u$ and location of spots for testing weak* $L^\infty$ convergence.
\label{fig:exact}}
\end{figure}

\bigskip

We performed test computations on a $2$d domain $\omega_c = \omega_o = \Omega = (0,1)^2$, on a regular finite element grid consisting of $2N^2$ triangles, with $N=32$. As values of the nonlinearity parameter, we considered $\kappa =1$ and $\kappa=100$, always in the left and right hand side of the page, respectively, in the figures and tables below. 
The exact source function $u_{ex}$ was chosen as a characteristic function of a cylinder of height $1$,  centered at $(0.3, 0.3)$ with radius $0.2$, where the exact regularization parameter $\rho^\dagger$ takes the value $1$, see Figure \ref{fig:exact}, left. 

In order to avoid an inverse crime, we generated the synthetic data on a finer grid and, after projection of $u_{ex}$ onto the computational grid, we added normally distributed random noise of levels $\delta$ in $\{0.1 \%, 1\%, 5\%, 10\%\}$ to obtain synthetic data 
\revision{$g^\delta$.}
\revmarg{2, 12.}

In all tests we started with the constant function with value zero for $u_0$. Moreover, we set $\tau = 2.0$ and $\tilde{\theta} = 0.51, \tilde{\Theta} = 0.98$. We always started the algorithm with a very coarse approximation for $\rho$, by setting $\rho_{start} = 100$.  According to our convergence result with $\mathcal{R}=\|.\|_{L^\infty(\Omega)}$, we can expect weak* convergence in $L^\infty(\Omega)$ here. Thus, we computed the errors in certain spots within the two homogeneous regions and on their interface,
\revision{
\[
\begin{aligned}
&\mbox{spot}_1 = [0.3,0.3+\tfrac{1}{N})\times[0.3,0.3+\tfrac{1}{N}), \\
&\mbox{spot}_2 = [0.7,0.7+\tfrac{1}{N})\times[0.3,0.3+\tfrac{1}{N}), \\ 
&\mbox{spot}_3 = [0.3,0.3+\tfrac{1}{N})\times[0.5,0.5+\tfrac{1}{N}), 
\end{aligned}
\]
see Figure \ref{fig:exact}, right, more precisely, corresponding to 
the chracteristic functions of $\mbox{spot}_i$ considered as elements of $L^1(\Omega)$,
in order to exemplarily test weak* $L^\infty(\Omega)(=L^1(\Omega)^*)$ convergence.
}
\revmarg{2, 10.} 
Additionally we computed $L^1$ errors.
In our tests we actually used the constraints $0\leq u\leq\rho$, which corresponds to defining $\mathcal{R}=\|\cdot\|_\infty+I_{\geq0}$, with the indicator function $I_{\geq0}(u)=\begin{cases}0\mbox{ if }u\geq0\mbox{ a.e. on }\Omega\\+\infty\mbox{ else}\end{cases}$.

The results are documented as follow: First, Figure \ref{fig:uniform} displaying two columns with the reconstructions for $\kappa =1$ on the left and for $\kappa = 100$ on the right for noise levels $1.0\%, 5.0\%, 10.0\%$. 
Then, Table \ref{tab:uniform} with the results obtained by the tests for $\delta=0.1\%, 1.0\%, 5.0\%, 10.0\%$, showing how many Gauss Newton steps
\revision{
and how many minimizations} 
\revmarg{3, 3.}
were required, the finally computed regularization parameter $\rho$, the $L^1$ error, and the error at the three spots for particular values of $\delta$ and $\kappa$.

\begin{figure}[h]
\begin{tabular}{c|c}
$\kappa = 1$ & $\kappa = 100$ \\
\hline
$\delta = 1.0\%$ & $\delta = 1.0\%$ \\
\includegraphics[width=0.43\textwidth]{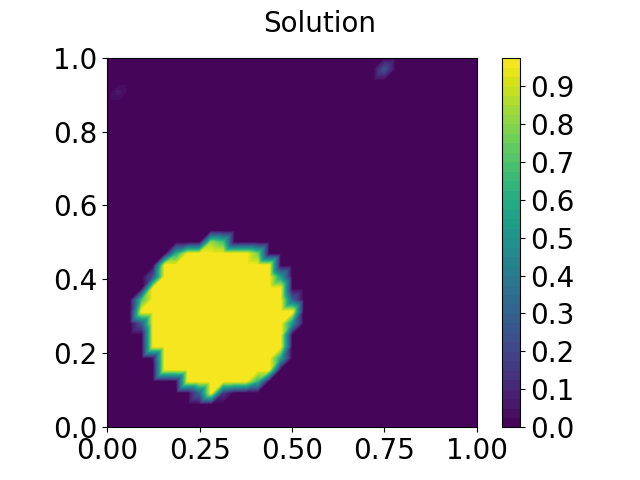} & \includegraphics[width=0.43\textwidth]{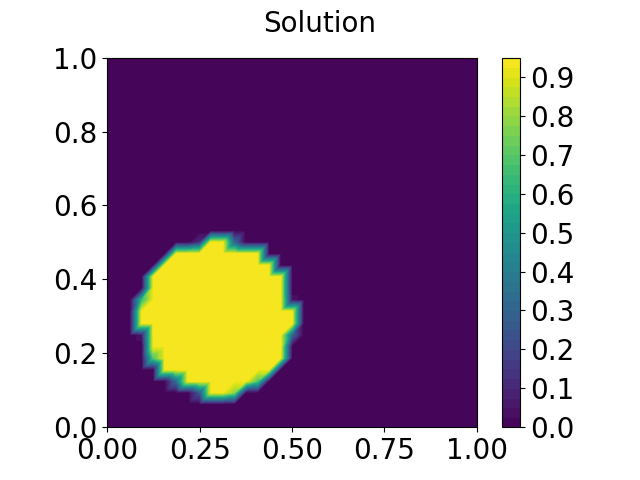} \\
\hline
$\delta = 5.0\%$ & $\delta = 5.0\%$ \\
\includegraphics[width=0.43\textwidth]{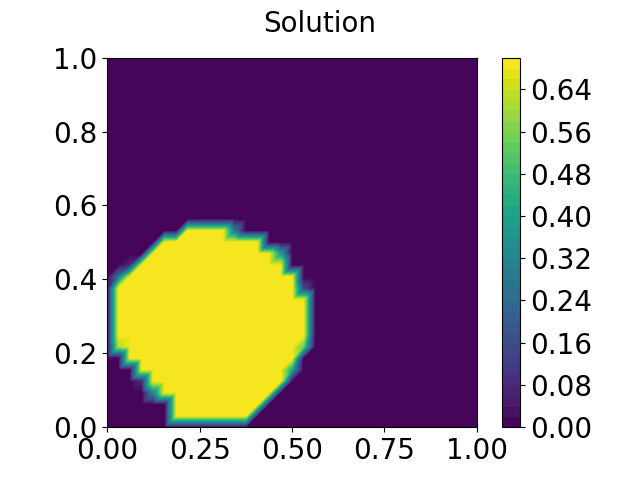} & \includegraphics[width=0.43\textwidth]{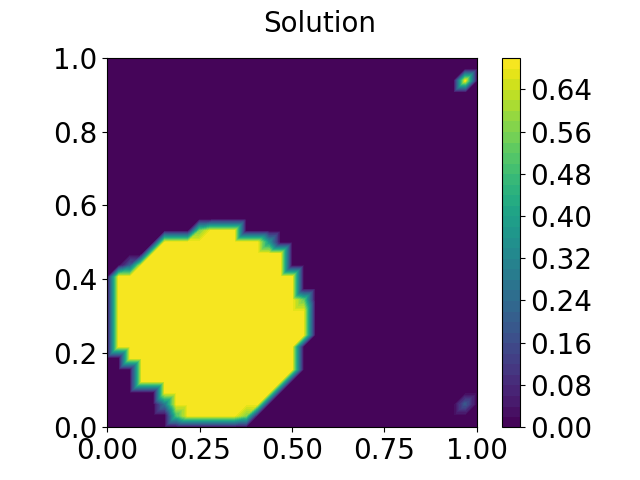} \\
\hline
$\delta = 10.0\%$ & $\delta = 10.0\%$ \\
\includegraphics[width=0.43\textwidth]{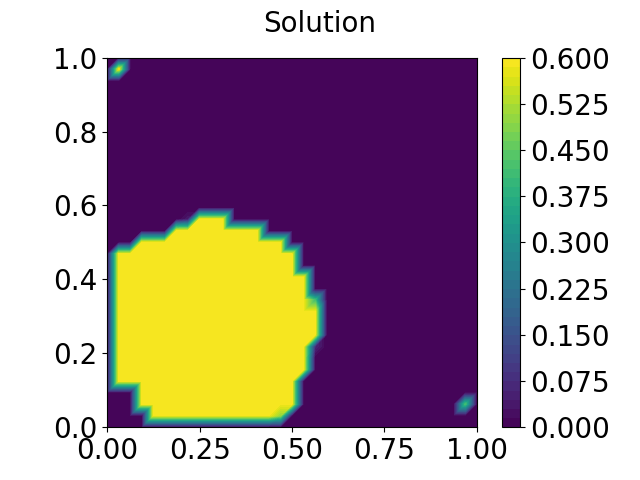} & \includegraphics[width=0.43\textwidth]{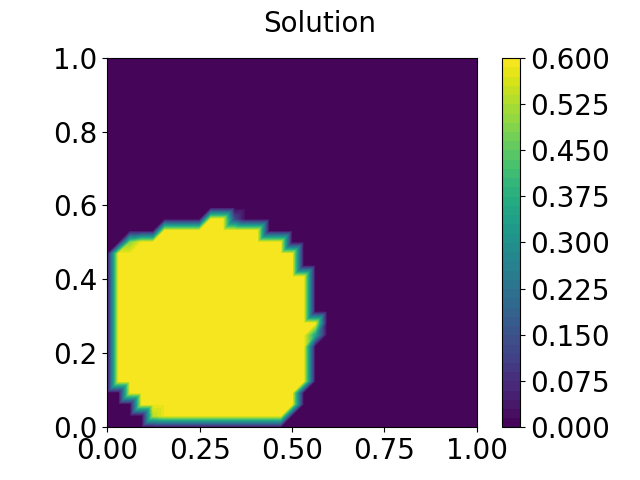} \\
\end{tabular}
\caption{Reconstructions of the source $u$.\label{fig:uniform}}
\end{figure}

{\small
\begin{table}[h]
\begin{tabular}
{|l||l|}
\hline
 $\kappa$ = 1& \\
\hline
& $ \delta = 0.1\%$ \\
\hline
\textit{GN iterations}& 15 \\ 
\hline
\revision{
\textit{minimizations}
}
& 
\revision{
74
}
\\ 
\hline
\textit{parameter} $\rho$& 0.9973 \\
\hline
$L^1(\Omega)$ error& 0.0151 \\
\hline
error point 1& 0.0679 \\
\hline
error point 2& 0.0000 \\
\hline
error point 3&0.2455 \\
\hline
\hline
& $\delta = 1.0\%$ \\
\hline
\textit{GN iterations}& 10 \\ 
\hline
\revision{
\textit{minimizations}
}
& 
\revision{
37
}
\\ 
\hline
\textit{parameter} $\rho$& 0.9233 \\
\hline
$L^1(\Omega)$ error& 0.0255 \\
\hline
error point 1& 0.0766\\
\hline
error point 2& 0.0000\\
\hline
error point 3&0.6574 \\
\hline
\hline
& $\delta = 5.0\%$ \\
\hline
\textit{GN iterations}& 7 \\ 
\hline
\revision{
\textit{minimizations}
}
& 
\revision{
23
}
\\ 
\hline
\textit{parameter} $\rho$& 0.6866 \\
\hline
$L^1(\Omega)$ error& 0.0968 \\
\hline
error point 1& 0.3133\\
\hline
error point 2& 0.0000\\
\hline
error point 3&0.6866 \\
\hline
\hline
& $\delta = 10.0\%$ \\
\hline
\textit{GN iterations}& 5\\ 
\hline
\revision{
\textit{minimizations}
}
& 
\revision{
16
}
\\ 
\hline
\textit{parameter} $\rho$& 0.5859  \\
\hline
$L^1(\Omega)$ error& 0.1296 \\
\hline
error point 1& 0.4140\\
\hline
error point 2& 0.0000\\
\hline
error point 3&0.5859 \\
\hline
\end{tabular} \hspace*{0.1\textwidth} \begin{tabular}
{|l||l|}
\hline
 $\kappa =$100 &\\
\hline
& $\delta = 0.1\%$ \\
\hline
\textit{GN iterations}& 15 \\ 
\hline
\revision{
\textit{minimizations}
}
& 
\revision{
74
}
\\ 
\hline
\textit{parameter} $\rho$& 0.9973 \\
\hline
$L^1(\Omega)$ error& 0.0151\\
\hline
error point 1&  0.0679\\
\hline
error point 2&  0.0000\\
\hline
error point 3& 0.2684\\
\hline
\hline
& $\delta = 1.0\%$ \\ 
\hline
\textit{GN iterations} & 10 \\ 
\hline
\revision{
\textit{minimizations}
}
& 
\revision{
36
}
\\ 
\hline
\textit{parameter} $\rho$&  0.9233 \\
\hline
$L^1(\Omega)$ error & 0.0256\\
\hline
error point 1&  0.0766\\
\hline
error point 2&  0.0000\\
\hline
error point 3 & 0.6574\\
\hline
\hline
& $\delta = 5.0\%$ \\
\hline
\textit{GN iterations} & 8 \\ 
\hline
\revision{
\textit{minimizations}
}
& 
\revision{
21
}
\\ 
\hline
\textit{parameter} $\rho$&  0.6591 \\
\hline
$L^1(\Omega)$ error & 0.1068\\
\hline
error point 1&  0.3408\\
\hline
error point 2&  0.0000\\
\hline
error point 3&0.6591\\
\hline
\hline
& $\delta = 10.0\%$ \\
\hline
\textit{GN iterations}& 5 \\ 
\hline
\revision{
\textit{minimizations}
}
& 
\revision{
16
}
\\ 
\hline
\textit{parameter} $\rho$& 0.5859 \\
\hline
$L^1(\Omega)$ error & 0.1320\\
\hline
error point 1& 0.4140\\
\hline
error point 2& 0 0.0000\\
\hline
error point 3& 0.5859\\
\hline
\end{tabular}
\caption{Results obtained from experiments without adaptive mesh refinement.
\label{tab:uniform}}
\end{table}

}

In order to numerically illustrate Corollary \ref{cor:discr}, we have implemented the IRGNM with an adaptive mesh refinement using the error estimators from \cite{ClasonKaltenbacherWachsmuth16}.
These estimators have been shown to reliably estimate the discretization error in the cost functional corresponding to the linear problem even in the setting of general nonreflexive Banach spaces relevant here, see \cite[Section 4]{ClasonKaltenbacherWachsmuth16} so we applied them to estimate the error in the Newton step cost function as a replacement for \eqref{alp}. More precisely, in an inner loop, we adaptively refine the discretization (marking and subdividing those elements with the largest contribution in the error estimator) until the error estimator falls below a fixed multiple of the computed nonlinear discrepancy $\|F^k_h(u_{k,h}^\delta)-g^\delta\|$.

In order to obtain error estimators for the quantities of interest in \eqref{eta}--\eqref{alp}, the concept of goal oriented error estimators would have to be employed. Indeed, in \cite{HintermullerHoppe08,VexlerWollner08} such estimators based on adjoint techniques (dual weighted residual estimators) have been developed in the setting of control constraints as relevant here. However, they contain an $L^2$ regularization term for the control $u$ (there denoted by $q$) which is not present in our setting and which is crucial for the computation of the discretized control (cf. \cite[Equation (4.7)]{VexlerWollner08}). Thus in order to follow the approach from \cite{VexlerWollner08}, one would have to modify the arguments there in order to prove that the discrete version of $q$, that we obtain more implicitely from the first order optimality conditions, still allows to establish the statements made about the error estimator there. This will be subject of future work.

The tests here started with $N=8$ for the triangulation of finite elements and we show the results in the same way as before but also displaying the adaptively refined mesh, see Figure \ref{fig:adaptive} and Table \ref{tab:adaptive}.

\begin{figure}[h]
\begin{tabular}{c|c}
$\kappa = 1$ & $\kappa = 100$ \\
\hline
$\delta = 1.0\%$ & $\delta = 1.0\%$ \\
\includegraphics[width=0.43\textwidth]{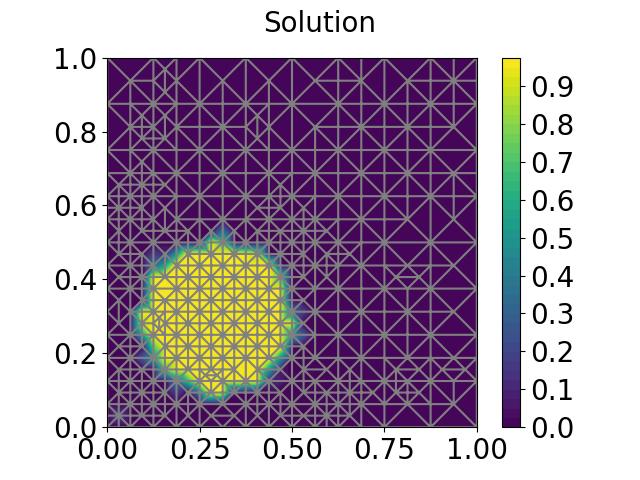} & \includegraphics[width=0.43\textwidth]{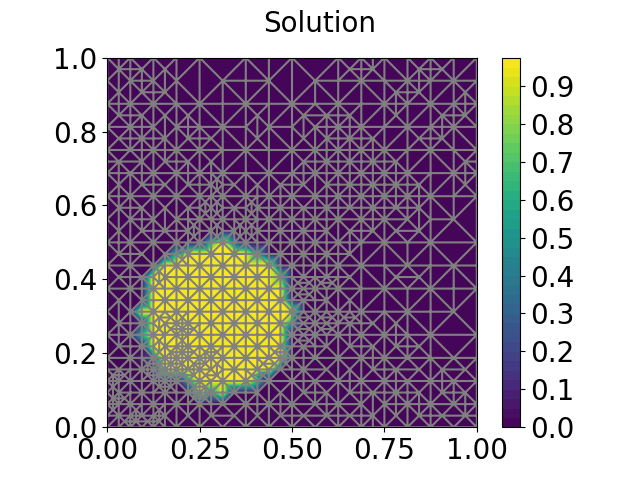} \\
\hline
$\delta = 5.0\%$ & $\delta = 5.0\%$ \\
\includegraphics[width=0.43\textwidth]{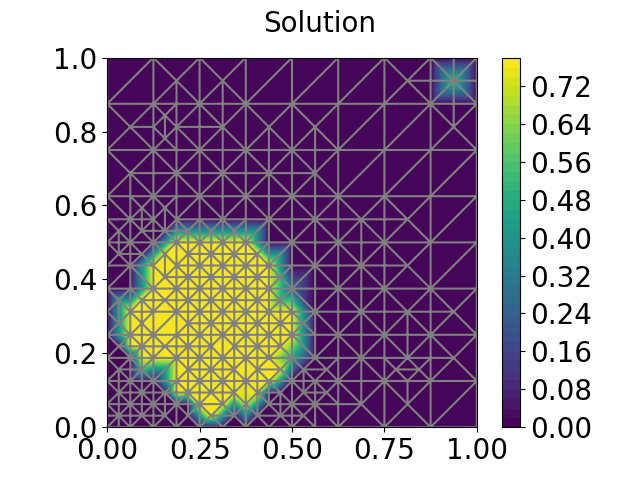} & \includegraphics[width=0.43\textwidth]{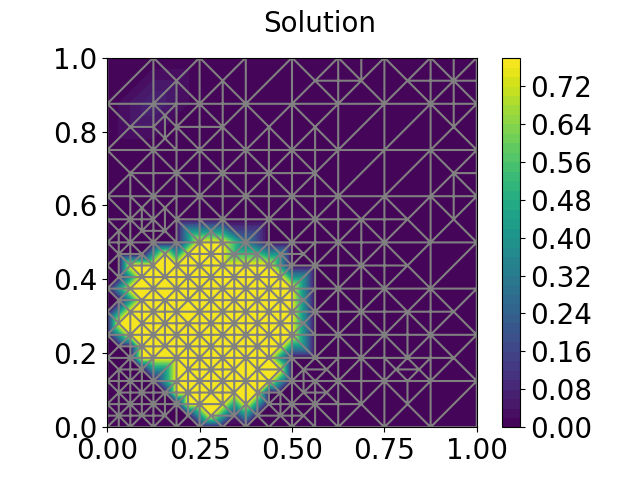} \\
\hline
$\delta = 10.0\%$ & $\delta = 10.0\%$ \\
\includegraphics[width=0.43\textwidth]{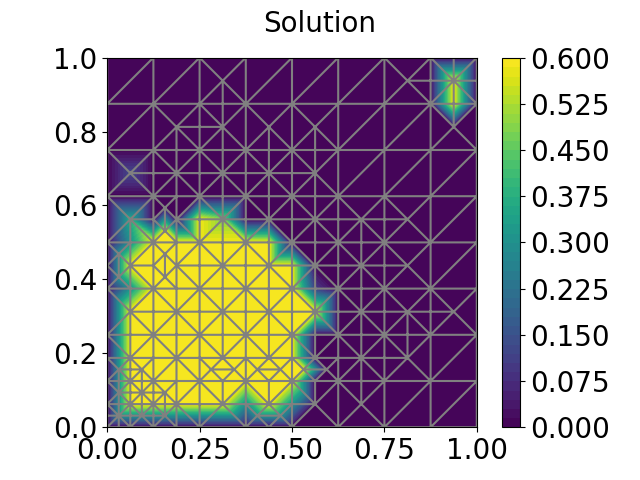} & \includegraphics[width=0.43\textwidth]{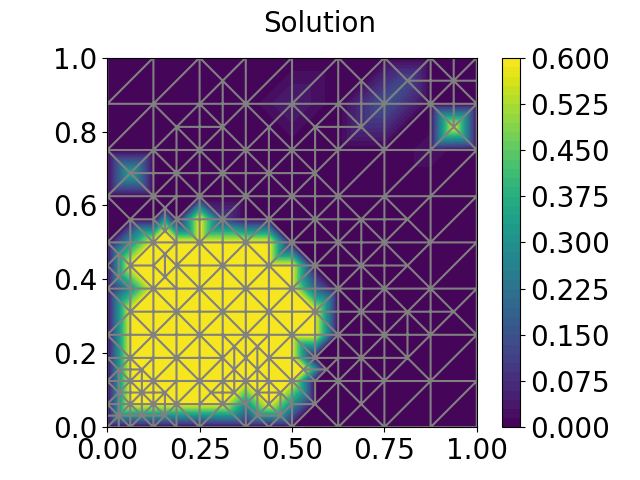} \\
\end{tabular}
\caption{Reconstructions of the source $u$ with adaptive mesh refinement.
\label{fig:adaptive}}
\end{figure}

{\small
\begin{table}[h]
\begin{tabular}
{|l||l|}
\hline
 $\kappa = $ 1 & \\
\hline
& $\delta = 0.1\%$  \\
\hline
\textit{GN iterations}& 22 \\ 
\hline
\revision{
\textit{minimizations}
}
& 
\revision{
70
}
\\ 
\hline
\textit{parameter} $\rho$& 0.9934  \\
\hline
$L^1(\Omega)$ error& 0.0007 \\
\hline
error point 1& 0.0065 \\
\hline
error point 2& 0.0000\\
\hline
error point 3&0.2155 \\
\hline
\hline
& $\delta = 1.0\%$ \\
\hline
\textit{GN iterations}& 11 \\ 
\hline
\revision{
\textit{minimizations}
}
& 
\revision{
40
}
\\ 
\hline
\textit{parameter} $\rho$& 0.9233 \\
\hline
$L^1(\Omega)$ error& 0.0390 \\
\hline
error point 1& 0.0766\\
\hline
error point 2& 0.0000\\
\hline
error point 3&0.6574 \\
\hline
\hline
& $\delta = 5.0\%$\\
\hline
\textit{GN iterations}& 7\\ 
\hline
\revision{
\textit{minimizations}
}
& 
\revision{
24
}
\\ 
\hline
\textit{parameter} $\rho$& 0.7629 \\
\hline
$L^1(\Omega)$ error& 0.1161\\
\hline
error point 1& 0.2370\\
\hline
error point 2& 0.0000\\
\hline
error point 3&0.6530\\
\hline
\hline
& $\delta = 10.0\%$ \\
\hline
\textit{GN iterations}& 5\\ 
\hline
\revision{
\textit{minimizations}
}
& 
\revision{
17
}
\\ 
\hline
\textit{parameter} $\rho$& 0.5859  \\
\hline
$L^1(\Omega)$ error& 0.1545 \\
\hline
error point 1& 0.4140\\
\hline
error point 2& 0.0000\\
\hline
error point 3&0.5356\\
\hline
\end{tabular} \hspace*{0.1\textwidth} \begin{tabular}
{|l||l|}
\hline
 $\kappa = $  100 &\\
\hline
& $\delta = 0.1\%$  \\
\hline
\textit{GN iterations}& 27 \\ 
\hline
\revision{
\textit{minimizations}
}
& 
\revision{
84
}
\\ 
\hline
\textit{parameter} $\rho$& 0.9973 \\
\hline
$L^1(\Omega)$ error & 0.0042\\
\hline
error point 1&  0.0072\\
\hline
error point 2&  0.0000\\
\hline
error point 3& 0.0000\\
\hline
\hline
& $\delta = 1.0\%$ \\
\hline
\textit{GN iterations}& 12 \\ 
\hline
\revision{
\textit{minimizations}
}
& 
\revision{
42
}
\\ 
\hline
\textit{parameter} $\rho$&  0.9400 \\
\hline
$L^1(\Omega)$ error& 0.0349\\
\hline
error point 1& 0.0595\\
\hline
error point 2& 0.0000\\
\hline
error point 3& 0.6696\\
\hline
\hline
& $\delta = 5.0\%$ \\
\hline
\textit{GN iterations}& 8 \\ 
\hline
\revision{
\textit{minimizations}
}
& 
\revision{
24
}
\\ 
\hline
\textit{parameter} $\rho$ & 0.7629 \\
\hline
$L^1(\Omega)$ error& 0.1199\\
\hline
error point 1&  0.2370\\
\hline
error point 2&  0.0000\\
\hline
error point 3& 0.7322\\
\hline
\hline
& $\delta = 10.0\%$ \\
\hline
\textit{GN iterations} & 5 \\ 
\hline
\revision{
\textit{minimizations}
}
& 
\revision{
17
}
\\ 
\hline
\textit{parameter} $\rho$&  0.5859 \\
\hline
$L^1(\Omega)$ error& 0.602\\
\hline
error point 1& 0.4140\\
\hline
error point 2&  0.0000\\
\hline
error point 3& 0.5489\\
\hline
\end{tabular}
\caption{Results obtained from experiments with adaptive mesh refinement.
\label{tab:adaptive}}
\end{table}
}

\section{Conclusions and remarks}\label{sec:concl}
In this paper we have proposed and analyzed an Ivanov regularized Gauss-Newton methods with an a posteriori choice of the regularization radius. Our analysis works in general nonreflexive Banach spaces and therefore also comprises $L^\infty(\Omega)$ as a preimage space, which leads to a bound constrained quadratic minimization problem in each Newton step. This setting is illustrated by numerical experiments for a nonlinear inverse source problem, also giving some first results with an adaptively refined computational mesh.
Further investigations on goal oriented adaptivity will be subject of future research.

\providecommand{\bysame}{\leavevmode\hbox to3em{\hrulefill}\thinspace}
\providecommand{\MR}{\relax\ifhmode\unskip\space\fi MR }
\providecommand{\MRhref}[2]{%
  \href{http://www.ams.org/mathscinet-getitem?mr=#1}{#2}
}
\providecommand{\href}[2]{#2}

\end{document}